\documentclass[10pt,final,leqno,onefignum,onetabnum,hidelinks]{siamltex}
\usepackage{amssymb,amsfonts,amsmath,bm}
\usepackage{hyperref}
\usepackage{graphicx}
\usepackage{color}
\usepackage{eepic,epic}{\tiny {\scriptsize {\small {\large }}}}
\usepackage{epsfig,subfigure,epstopdf}
\usepackage[showonlyrefs]{mathtools}
\usepackage[notref,notcite]{showkeys}
\usepackage{mathrsfs}

\newtheorem{remark}{Remark}[section]
\newtheorem{assumption}{Assumption}[section]
\definecolor{darkred}{rgb}{0.85,0,0}

\definecolor{green}{rgb}{0,0.7,0}

\def\R{{\mathbb R}}

\begin{document}
	
	\title{Analysis and approximation to parabolic optimal control problems with measure-valued controls in time}
	\author{Wei Gong\thanks{The State Key Laboratory of Scientific and Engineering Computing, Institute of Computational Mathematics \& National Center for Mathematics and Interdisciplinary Sciences, Academy of Mathematics and Systems Science, Chinese Academy of Sciences, 100190 Beijing, China. Email address: wgong@lsec.cc.ac.cn}
		\and Dongdong Liang\thanks{The Hong Kong Polytechnic University Shenzhen Research Institute, Shenzhen 518057, China. Email address: dongdong.liang@polyu.edu.hk}
	}
	
	\date{\today}
	
	\maketitle
	\begin{abstract}
		In this paper, we investigate an optimal control problem governed by parabolic equations with measure-valued controls over time. We establish the well-posedness of the optimal control problem and derive the first-order optimality condition using Clarke's subgradients, revealing a sparsity structure in time for the optimal control. Consequently, these optimal control problems represent a generalization of impulse control for evolution equations. To discretize the optimal control problem, we employ the space-time finite element method. Here, the state equation is approximated using piecewise linear and continuous finite elements in space, alongside a Petrov-Galerkin method utilizing piecewise constant trial functions and piecewise linear and continuous test functions in time. The control variable is discretized using the variational discretization concept. For error estimation, we initially derive a priori error estimates and stabilities for the finite element discretizations of the state and adjoint equations. Subsequently, we establish weak-* convergence for the control under the norm $\mathcal{M}(\bar I_c;L^2(\omega))$, with a convergence order of $O(h^\frac{1}{2}+\tau^\frac{1}{4})$ for the state.
		
\vspace{0.2cm}		
	\end{abstract}

	\section{\bf Introduction}\label{Se:intr}	
	{Let $\Omega\subset\R^d$ ($d=2,3$) be a convex polyhedron with boundary $\Gamma:=\partial\Omega$, and $I := (0, T)$ with $T > 0$. In this paper we consider the following optimal control problem: 
		\begin{align}\label{min-J}
			\min_{(u,q)\in X\times\mathcal{M}(\bar I_c;L^2(\omega))}
			J(u,q)
			=\frac{1}{2}\|u-u_d\|_{L^2(I;L^2(\Omega))}^2
			+\frac{\beta}{2}\|u(T)-u_T\|_{L^2(\Omega)}^2
			+\alpha\|q\|_{\mathcal{M}(\bar I_c;L^2(\omega))},
		\end{align}
		where $\mathcal{M}(\bar{I}_c;L^2(\omega))$ is the control space of vector measures that will be defined in the subsequent section, $ X:=\{v\in L^2(I;L^2(\Omega)),\ v(T)\in L^2(\Omega)\}$ is the observation space, $u_d\in L^2(I;L^2(\Omega))$ and $u_T\in L^2(\Omega)$ are given observations or target states,  $\alpha>0$ is a regularization parameter, $\beta\geq 0$ is a weight parameter. The state $u$ and the control $q\in \mathcal{M}(\bar{I}_c;L^2(\omega))$ in \eqref{min-J} are constrained by the following parabolic equation with initial data $u_0\in L^2(\Omega)$ and source $f\in L^2(I;L^2(\Omega))$:
		\begin{align}\label{PDE_state}
			\left\{
			\begin{aligned}
				\partial_t u - \Delta u & = f+\chi_{I_c\times\omega} q
				&&\mbox{in}\,\,\,\Omega\times I,\\
				u&=0 && \mbox{on}\,\, \Gamma\times I,\\
				u(0)&=u_0 && \mbox{in}\,\,\Omega ,
			\end{aligned}
			\right.
		\end{align}
		where $\omega\subset\Omega$ and $I_c\subset\subset I$ (relatively compact) denote the spatial and temporal control domains, respectively. Here $\chi_{I_c\times\omega}$ is the characteristic function of $I_c\times\omega$ taking values $1$ in $I_c\times\omega$ and $0$ otherwise, which can be viewed as a zero extension operator.  }
	
		In recent years, sparse controls of partial differential equations have garnered significant attention. Initially motivated by actuator placement, their application scope has since expanded to practical domains. Two main approaches have emerged to achieve sparsity: one involves $L^1$-norm regularization in the objective functional, while the other relies on measure-valued controls. The pioneering work in this direction is \cite{Stadler2009}, which investigated $L^1(\Omega)$ control for linear elliptic equations. Subsequently, \cite{CasasHerzogWachsmuth-2017} addressed the spatio-temporally sparse optimal control problem of semilinear parabolic equations, introducing three different sparsity-promoting terms in the objective functional: $L^1(\Omega\times I)$, $L^2(I;L^1(\Omega))$, or $L^1(\Omega;L^2(I))$. For the latter case, \cite{CasasMateosRosch-2017} provided an error estimate for its fully discrete finite element approximation, with further refinement seen in \cite{CasasMateosRosch-2018}. Additionally, \cite{HerzogStadlerWachsmuth2012} explored directional sparse control for parabolic equations, where controls exhibit sparsity in space but not necessarily in time. Notably, the sparsity pattern remains constant over time.
	
For measure-valued control problems, \cite{CasasClasonKunisch-2012} investigates elliptic equations with control space $\mathcal{M}(\Omega)$, while \cite{CasasClasonKunisch-2013} addresses parabolic equations in $L^2(I;\mathcal{M}(\Omega))$, providing error estimates for finite element approximations. For parabolic control problems in space-time measure $\mathcal{M}(I\times\Omega)$, \cite{CasasKunisch-2016} is a relevant reference. Extending the directional sparsity concept (\cite{HerzogStadlerWachsmuth2012}) to measure spaces, \cite{KunischPieperVexler-2014} examines measure-valued directional sparsity for parabolic control problems with control space $\mathcal{M}(\Omega;L^2(I))$, deriving an a priori error estimate. Optimal control of the linear second-order wave equation with measure-valued controls in $\mathcal{M}(\Omega;L^2(I))$ is discussed in \cite{KunischTrautmannVexler-2016}. In \cite{TrautmannVexlerZlotnik-2018}, the authors explore measure-valued optimal control problems for 1D wave equations with control spaces of either measure-valued functions $L_{w^*}^2(I;\mathcal{M}(\Omega))$ or vector measures $\mathcal{M}(\Omega,L^2(I))$, deriving error estimates for the optimal state variable and the error measured in the cost functional. Additionally, \cite{HerbergHinzeSchumacher2020} investigates a variational discretization of a parabolic optimal control problem with space-time measure controls, employing a Petrov-Galerkin method with piecewise constant states and piecewise linear and continuous test functions in time for temporal discretization of the state equation. References \cite{CasasVexlerZuazua-2015,CasasKunisch-2019,VexlerLeykekhmanWalter2020,HerbergHinze2020} provide further insights into initial value identification of parabolic equations in measure spaces.	
For time-dependent systems, the control problem (\ref{min-J})-(\ref{PDE_state}) posed in $\mathcal{M}(\bar I_c;L^2(\omega))$ yields controls with compact support in time. This characteristic allows for determining the optimal moments for control device actions, akin to a generalization of impulse control \cite{Bensoussan-Lions-1984,Chebotarev-2009,DuanWang-2020,DuanWangZhang-2019,Phung-Wang-Xu-2017,Qin-Wang-2017,Trelat-Wang-Zhang-2016,Yan-2016,Yang-2001,Yong-Zhang-1992,YuHuangLiu-2019}. Recall that in impulse control problems, the control $q$ in (\ref{min-J})-(\ref{PDE_state}) is replaced by
\begin{eqnarray}
q(x,t)=\sum\limits_{i=1}^m q_i(x)\otimes\delta_{\tau_i}(t),\nonumber
\end{eqnarray}
where $\delta_{\tau_i}$ denotes the Dirac delta measure concentrated at $\tau_i\in I_c\subset\subset I$, and $q_i{\otimes}\delta_{\tau_i}$ are linear functionals on space $C(\bar{I}_c;L^2(\omega))$, with $q_i\in L^2(\omega), i=1,2,\cdots, m$. Here, the impulse strengths $q_i,,i=1,2,\cdots,m,$ are optimized at prescribed time nodes $\tau_i\in (0,T)$ \cite{Gong-2013}. However, in many cases, the interest lies in optimizing both the time nodes and the impulse strengths. This motivation leads to the formulation of the generalized impulse control problem as described by (\ref{min-J})-(\ref{PDE_state}). We remark that impulse control belongs to a class of important control and has wide applications (see, for instance, \cite{Bensoussan-Lions-1984,Yang-2001}). In many cases impulse control is an interesting alternative to deal with systems that cannot be acted on by means of continuous control inputs.

	The contributions of this article are threefold. Firstly, we investigate the well-posedness of both the state equation and the optimal control problem. Additionally, we derive the first-order optimality condition, revealing that the optimal control exhibits a sparsity structure independent of space. Secondly, for the state equation approximation, we utilize piecewise linear and continuous functions in space and a Petrov-Galerkin scheme from \cite{DanielsHinzeVierling2015} in time. Specifically, we employ piecewise constant trial functions and piecewise linear and continuous test functions. We also employ a variational discretization concept for the control. Lastly, we provide an a priori error estimate for the finite element approximation of the control problem. Building upon the a priori error and stability estimates for finite element discretizations of the state and adjoint equations, we establish a convergence order of $O(h^\frac{1}{2}+\tau^\frac{1}{4})$ for the approximation of the state. Furthermore, we demonstrate weak-* convergence for the control under the $\mathcal{M}(\bar I_c;L^2(\omega))$ norm, with a convergence order of $O(h+\tau^\frac{1}{2})$ for the discrete cost functional.

	The remainder of this paper is organized as follows. Section 2 presents some preliminary results, including the definition of very weak solutions to the state equation, as well as the global and local regularity and weak-* continuity of the state variable. In Section 3, we derive the first-order optimality system and investigate the sparse structure and regularity of the optimal control. Two discrete optimal control problems and their associated optimality systems are provided in Section 4. Section 5 primarily focuses on the error analysis for the optimal control problem.

	\section{Preliminaries}
	\subsection{Notations for function spaces}
	Let $W^{k,p}(\Omega)$ $(k\in N_+\cup \{0\},\,1\le p\le \infty)$ be the usual Sobolev space defined in $\Omega$ with the norm $\Vert\cdot\Vert_{W^{k,p}(\Omega)}$. Note that  $W_0^{k,p}(\Omega)$ $(H_0^k(\Omega))$ is the closed subspace of $W^{k,p}(\Omega)$ $(H^k(\Omega))$ with null-traces on the boundary $\Gamma$. We abbreviate it by $H^k(\Omega):=W^{k,2}(\Omega)$ (resp. $H_0^k(\Omega):=W_0^{k,2}(\Omega)$) $( k\ge1)$  with norms $\Vert\cdot\Vert_{H^k(\Omega)}$, and $L^p(\Omega):=W^{0,p}(\Omega)$ that is the $p$-integrable function space in $\Omega$ with norms $\Vert\cdot\Vert_{L^p(\Omega)}$. Particularly, $L^2(D)$ $(D=\Omega,\omega)$ is a Hilbert space with inner products $(\cdot,\cdot)$ and norms $\|\cdot\|$.
	Let ${C}(\bar I_c)$ be the Banach space consisting of continuous functions on $\bar{I}_c$ equipped with the supremum norm $\|\cdot\|_{C(\bar{I}_c)}$. Let $\mathcal{M}(\bar I_c)$ be the dual space of ${C}(\bar I_c)$ that is a Banach space under the norm
	\begin{eqnarray}
		\|v\|_{\mathcal{M}(\bar I_c)}:=\sup\left\{\int_{\bar I_c}wdv,\,\forall \omega\in {C}(\bar I_c)\,,\,\|w\|_{{C}(\bar I_c)}\leq 1\right\}\quad\forall v\in \mathcal{M}(\bar I_c),\nonumber
	\end{eqnarray}
	which can be identified with the space of regular Borel measures in $I_c$.

	For a given positive measure $\mu\in \mathcal{M}(\bar{I}_b)$ the notation $L^p(I_b,\mu;L^2(\Omega))$ $(1\le p)$ denotes the set of all functions defined on a subset $I_b\subset I$ and valued in $L^2(\Omega)$, which is a Banach space endowed with the norm
	$$
	\Vert v\Vert_{L^p(I_b,\mu;L^2(\Omega))}:=\Big(\int_{I_b}\left\Vert v(t)\right\Vert^{p}d\mu(t)\Big)^{\frac{1}{p}}
	$$
	for $p<\infty$, and
	$$
	\Vert v\Vert_{L^\infty(I_b,\mu;L^2(\Omega))}:=\mathop{\rm ess\ sup}_{t\in I_b}\left\Vert v(t)\right\Vert 
	$$
	for $p=\infty$. If $\mu$ is a Lebesgue measure, we abbreviate $L^p(I_b,\mu;L^2(\Omega))$ as $L^p(I_b;L^2(\Omega))$.

	The space $L^2(\omega;\mathcal{M}(\bar I_c))$, consisting of all weakly-* measurable functions $q:\omega\rightarrow \mathcal{M}(\bar I_c)$, is a Banach space endowed with the norm 
	\begin{eqnarray}
		\|q\|_{L^2(\omega;\mathcal{M}(\bar I_c))}:=\Big(\int_\omega \|q(x)\|^2_{\mathcal{M}(\bar I_c)}dx\Big)^{1\over 2}\quad\forall q\in L^2(\omega;\mathcal{M}(\bar I_c))\nonumber
	\end{eqnarray}
	which can be identified with the dual of $L^2(\omega;{C}(\bar I_c))$, where $L^2(\omega;{C}(\bar I_c))$ denotes the Banach space consisting of all functions defined on $\omega$ and valued in ${C}(\bar I_c)$ with the norm
	\begin{equation}
		\Vert v\Vert_{L^2(\omega;{C}(\bar I_c))}:=\Big(\int_\omega\Vert v(x)\Vert_{{C}(\bar I_c)}^2dx\Big)^\frac{1}{2}\quad\forall v\in L^2(\omega;{C}(\bar I_c)).\nonumber
	\end{equation}
	
	For any given Banach space $X$, e.g., $H^1_0(\Omega)$, $L^2(\Omega)$, etc.,  the notation $C(\bar{I}_b;X)$ denotes the set of all continuous functions on $\bar{I}_b$ and valued in $X$,  which is a Banach space under the supremum norm 
	$$\Vert v\Vert_{C(\bar{I}_b;X)}:=\mathop{\rm sup}\limits_{t\in \bar{I}_b}\left\Vert v(t)\right\Vert_X\quad\forall v\in C(\bar{I}_b;X).$$
Then, we define $\mathcal{M}(\bar I_c;L^2(\omega))$ as the space containing all countably additive measures with bounded total variations defined on the Borel sets $\mathcal{B}(\bar I_c)$ and valued in $L^2(\omega)$. For any $\mu\in \mathcal{M}(\bar I_c;L^2(\omega))$, the variation measure $|\mu|\in \mathcal{M}(\bar I_c)$ is defined as
	\begin{eqnarray}
		|\mu|(B):=\sup \Big\{\sum\limits_{n=1}^\infty\|\mu(B_n)\|_{L^2(\omega)} :\ \{B_n\}_{n=1}^\infty\subset \mathcal{B}(\bar I_c)\ \mbox{is\ the\ disjoint\ partition\ of}\ B\Big\}\nonumber
	\end{eqnarray}
	for any $B\in \mathcal{B}(\bar I_c)$, where $\mathcal{B}(\bar I_c)$ denotes the Borel set on $\bar I_c$. We denote by $|\mu|(\bar I_c)$ the total variation of $\mu$. The space $\mathcal{M}(\bar I_c;L^2(\omega))$ endowed with the norm $\|\mu\|_{\mathcal{M}(\bar I_c;L^2(\omega))}=\||\mu|\|_{\mathcal{M}(\bar I_c)}=|\mu|(\bar I_c)$ is a Banach space (cf. \cite{KunischPieperVexler-2014} and \cite[Chap 12, Sec. 3]{Lange-1983}), and that can be identified to the dual of ${C}(\bar I_c;L^2(\omega))$.
	In the following we denote by $\langle\cdot,\cdot\rangle$ the duality pairing between $\mathcal{M}(\bar I_c)$ and ${C}(\bar I_c)$, $\mathcal{M}(\bar I_c;L^2(\omega))$ and ${C}(\bar I_c;L^2(\omega))$, respectively. 
	
	For each $\mu\in \mathcal{M}(\bar I_c;L^2(\omega))$, the polar decomposition of $\mu$ consisting of the variation measure $|\mu|\in \mathcal{M}(\bar I_c)$ and a space-time function $\mu'\in L^1(I_c,|\mu|;L^2(\omega))$, where the temporal support of $\mu^\prime$ is included in the support of $|\mu|$ (cf. \cite{KunischPieperVexler-2014,Lange-1983}), such that 
	\begin{eqnarray}
		d\mu =\mu'd|\mu|(t)\quad\mbox{\rm and}\quad \langle \mu ,w\rangle_{\bar I_c\times \omega}=\int_{ I_c}(\mu'(t),w(t))d|\mu|(t)\nonumber
	\end{eqnarray}
	for any $\mu\in \mathcal{M}(\bar I_c;L^2(\omega))$ and $w\in {C}(\bar I_c;L^2(\omega))$. Furthermore, similar to equation (2.3) in \cite{KunischPieperVexler-2014} we can show that $\mu'\in L^\infty(I_c,|\mu|;L^2(\omega))$ with $\|\mu'\|_{L^\infty(I_c,|\mu|;L^2(\omega))}\leq 1$ and (cf. \cite{KunischPieperVexler-2014,Lange-1983})
	\begin{eqnarray}
		\|\mu'(t)\|_{L^2(\omega)}= 1\quad \mbox{for}\ |\mu|-\mbox{almost\  all}\ t\in \bar I_c.\label{Polar}
	\end{eqnarray}  
	
	\begin{remark}
		Since $L^2(\omega;{C}(\bar I_c))\hookrightarrow {C}(\bar I_c;L^2(\omega))$ by Minkowski's inequality, we have $\mathcal{M}(\bar I_c;L^2(\omega))\hookrightarrow L^2(\omega;\mathcal{M}(\bar I_c))$. The difference is that $L^2(\omega;\mathcal{M}(\bar I_c))$ yields measure valued functions whose temporal supports are spatial dependent, while the ones for the former are spatial independent (cf. \cite{KunischPieperVexler-2014}). Based on the above embedding, we see that for each $\mu\in \mathcal{M}(\bar I_c;L^2(\omega))$, the representation $\mu(x)\in \mathcal{M}(\bar I_c)$  is well-defined for almost all $x\in\omega$. Since $\mu'\in L^\infty(I_c,|\mu|;L^2(\omega))$, it thus also belongs to  $L^2(I_c,|\mu|;L^2(\omega))$, or equivalently, $L^2(\omega;L^2(I_c,|\mu|))$. Therefore, $\mu'(x)\in L^2(I_c,|\mu|)$ for a.e. $x\in \omega$. We can now write (cf. \cite[eq. (2.5)]{KunischPieperVexler-2014})
		\begin{eqnarray}
			d\mu(x)=\mu'(x)d|\mu|\quad \mbox{ a.e.}\ x\in \omega.\label{measure_space_form}
		\end{eqnarray} 
	\end{remark}	
	\begin{lemma}\label{stateEexist1}
		For given $g\in L^2(I;H^{-1}(\Omega))$ and $z_T\in L^2(\Omega)$, there exists a unique solution $z\in L^2(I;H^1_0(\Omega))\cap H^1(I;H^{-1}(\Omega))$ to the following problem
		\begin{align}\label{backward_PDE}
			\left\{
			\begin{aligned}
				-\partial_t z - \Delta z & = g
				&&\mbox{in}\,\,\,\Omega\times(0,T),\\
				z&=0 && \mbox{on}\,\, \Gamma\times (0,T),\\
				z(T)&=z_T && \mbox{in}\,\,\Omega.
			\end{aligned}
			\right.
		\end{align}
		Moreover, the following estimate holds
		\begin{equation}\label{estimate_1}
			\|z\|_{C(\bar{I};L^2(\Omega))}+\|\partial_t z\|_{L^2(I;H^{-1}(\Omega))}+\| z\|_{L^2(I;H^1_0(\Omega))}\leq C\left(\| g\|_{L^2(I;H^{-1}(\Omega))}+\|z_T\|_{L^2(\Omega)}\right).
		\end{equation}
		If, in addition, $g\in L^2(I;L^2(\Omega))\ \mbox{\rm and}\ z_T\in H^1_0(\Omega)$, then $
		z\in L^2(I;H^2(\Omega)\cap H^1_0(\Omega))\cap\, H^1(I;L^2(\Omega))\cap\,C(\bar{I};H^1_0(\Omega))\nonumber
		$
		and there holds the stability estimate
		\begin{equation}
			\|z\|_{C(\bar{I};H^1_0(\Omega))}+\|\partial_t z\|_{L^2(I;L^2(\Omega))}+\| z\|_{L^2(I;H^2(\Omega))}\leq C\left(\| g\|_{L^2(I;L^2(\Omega))}+\|z_T\|_{H^1(\Omega)}\right).
		\end{equation}
		The constants $C>0$ is independent of $g$ and $z_T$ in the above two stability estimates.
	\end{lemma}
	
	\begin{proof}
		The proof of the existence, uniqueness  and stability estimate \eqref{estimate_1} of solutions can be found in \cite[Theorem 5.1]{JEJRSC}. The improved regularity and stability estimate  are classical; see, e.g., \cite{LC}.
		The regularity $z\in C(\bar{I};H^1_0(\Omega))$ can be obtained from the fact that $L^2(I;H^2(\Omega)\cap H^1_0(\Omega))\cap H^1(I;L^2(\Omega))\hookrightarrow C(\bar{I};H^1_0(\Omega))$.
	\end{proof}
	
	\subsection{Well-posedness of the state equation}
	To begin with, we first investigate the well-posedness of the state equation (\ref{PDE_state}). The very weak solution of equation (\ref{PDE_state}) can be defined by transposition techniques (cf.~\cite{LionsMagenes-1972}), which will be given in the following.
	\begin{definition}\label{def_1}
		For any given $f\in L^2(I;L^2(\Omega))$, $q\in \mathcal{M}(\bar I_c;L^2(\omega))$ and $u_0\in L^2(\Omega)$, a function $u\in L^2(I;L^2(\Omega))$ is called the very weak solution of equation \eqref{PDE_state} if it satisfies
		\begin{equation}\label{identity_1}
			(u,g)_{L^2(I;L^2(\Omega))}=\int_I(f,z_g)_{L^2(\Omega)}dt+\langle q,z_g\rangle_{\bar I_c\times \omega}+(u_0,z_g(0))_{L^2(\Omega)}\quad\forall g\in L^2(I;L^2(\Omega)),
		\end{equation}
		where $z_g\in C(\bar{I}_c;L^2(\omega))$ is the solution to equation (\ref{backward_PDE}) with the right-hand side $g$ and  $z_T=0$, and $\langle \cdot,\cdot\rangle_{\bar I_c\times \omega}$ denotes the duality pairing between $\mathcal{M}(\bar I_c;L^2(\omega))$ and $C(\bar I_c;L^2(\omega))$.
	\end{definition}
	
	Since $z_T=0$ and $\Omega$ is convex, the solution of equation \eqref{backward_PDE} satisfies $z\in L^2(I;H^2(\Omega)\cap H^1_0(\Omega))\cap H^1(I;L^2(\Omega))\hookrightarrow C(\bar{I};H_0^1(\Omega))$ by Lemma \ref{stateEexist1}. Therefore, the above definition is well-defined.
	
	\begin{theorem}\label{Thm:existence_state}
		Assume that $f\in L^2(I;L^2(\Omega))$, $u_0\in L^2(\Omega)$ and $q\in \mathcal{M}(\bar I_c;L^2(\omega))$, then the parabolic equation (\ref{PDE_state}) admits a unique very weak solution $u\in L^2(I;L^2(\Omega))$. Furthermore, $u\in L^2(I;H_0^1(\Omega))\cap L^\infty(I;L^2(\Omega))$ and there holds the following estimate:
		\begin{align}
			\|u\|_{L^2(I;H_0^1(\Omega))}+\|u\|_{L^\infty(I;L^2(\Omega))}\leq C(\|f\|_{L^2(I;L^2(\Omega))}+\|q\|_{\mathcal{M}(\bar I_c;L^2(\omega))}+\|u_0\|_{L^2(\Omega)}),\label{state_stability}
		\end{align}
		where $C>0$ is a constant independent of $f$, $q$ and $u_0$.
		
		In addition, assume that $I_c=(t_1,t_2)$ with $0<t_1<t_2<T$, then there exist $\hat{t}$ and $\tilde{t}$ satisfying $t_2<\hat{t}<\tilde{t}<T$, i.e.,  $(\tilde{t},T)\subseteq (\hat{t},T)\subseteq (t_2,T)$, such that $u|_{(\tilde{t},T)}\in H^1((\tilde{t},T);L^2(\Omega))\cap L^2((\tilde{t},T);H^2(\Omega)\cap H^1_0(\Omega))\hookrightarrow C([\tilde{t},T];H^1_0(\Omega))$ and
		\begin{equation}\label{local_estimate}
			\begin{split}
				&\Vert {u}\Vert_{L^2((\tilde{t},T);H^2(\Omega)\cap H_0^1(\Omega))}+\Vert {u}\Vert_{H^1((\tilde{t},T);L^2(\Omega))}+\Vert u\Vert_{C([\tilde{t},T];H^1_0(\Omega))}
				\\
				\le &C(\Vert f\Vert_{L^2((\hat{t},T);L^2(\Omega))}+\Vert u\Vert_{L^2((\hat{t},T);L^2(\Omega))})\\
				\le &C(\Vert f\Vert_{L^2(I;L^2(\Omega))}+\Vert q\Vert_{\mathcal{M}(\bar{I}_c;L^2(\omega))}+\Vert u_0\Vert_{L^2(\Omega)}),
			\end{split}
		\end{equation}
		where $C>0$ is a constant independent of $f$, $q$ and $u_0$.
	\end{theorem}
	\begin{proof}
		The proof of the existence of a unique very weak solution $u\in L^2(I;H_0^1(\Omega))\cap L^\infty(I;L^2(\Omega))$  can be found in, e.g., \cite{Gong-Hinze-Zhou-2014}, \cite[Theorem 2.2]{Gong-2013}, \cite[Theorem 2.4]{Gong-Liu-Yan-2018} or \cite[Theorem 2.4]{Meidner-Rannacher-Vexler-2011} for measure data in $\mathcal{M}([0,T])$. Here we include a brief proof for completeness.
		
		Since the state equation is linear, it suffices to consider the case either $u_0=0$, $f=0$ or $q=0$. If $q=0$, $u_0\in L^2(\Omega)$ and $f\in L^2(I;L^2(\Omega))$, it is obvious that problem (\ref{PDE_state}) admits a unique weak solution $u\in L^2(I;H_0^1(\Omega))\cap L^\infty(I;L^2(\Omega))$ satisfying (cf. \cite{Lions-1971,Thomee-2006})
		\begin{eqnarray}
			\|u\|_{L^2(I;H_0^1(\Omega))}+\|u\|_{L^\infty(I;L^2(\Omega))}\leq C(\|f\|_{L^2(I;L^2(\Omega))}+\|u_0\|_{L^2(\Omega)}).\nonumber
		\end{eqnarray}
		
		Now we consider the case $u_0=0$ and $f=0$. Let $\{q_n\}_n\subset {C}(\bar I_c\times \bar\omega)$ be the sequence converging weakly to $q$ in $\mathcal{M}(\bar I_c;L^2(\omega))$ and satisfy
		\begin{eqnarray}
			\|q_n\|_{L^1(I_c;L^2(\omega))}\leq \|q\|_{\mathcal{M}(\bar I_c;L^2(\omega))}.\nonumber
		\end{eqnarray}
		Let $u_n$ be the solution of 
		\begin{align}\label{forward_PDE_state}
			\left\{
			\begin{aligned}
				\partial_t u_n - \Delta u_n & =\chi_{I_c\times\omega} q_n
				&&\mbox{in}\,\,\,\Omega\times(0,T],\\
				u_n&=0 && \mbox{on}\,\, \partial\Omega\times (0,T],\\
				u_n|_{t=0}&=0 && \mbox{in}\,\,\Omega,
			\end{aligned}
			\right.
		\end{align}
		then one has $u_n\in L^2(I; H_0^1(\Omega))\cap H^1(I;H^{-1}(\Omega))$. Let $z$ be the solution of problem (\ref{backward_PDE}) for given $g\in \mathcal{D}(I\times \Omega)$ and $z_T=0$, it follows from Lemma \ref{stateEexist1}  that $z\in {C}(\bar I; H_0^1(\Omega))$. 
		Then, using integration by parts we obtain
		\begin{eqnarray}
			\int_I\int_\Omega gu_ndxdt&=&\int_I\int_\Omega(-\partial_tz -\Delta z)u_ndxdt\nonumber\\
			&=&\int_{I_c} (q_n,z)_{L^2(\omega)}dt\nonumber\\
			&\leq&\|q_n\|_{L^1(I_c;L^2(\omega))}\|z\|_{L^\infty(I_c;L^2(\omega))}\nonumber\\
			&\leq&\|q\|_{\mathcal{M}(\bar I_c;L^2(\omega))}\|z\|_{L^\infty(I_c;L^2(\omega))}.
		\end{eqnarray}
		Combining the following standard estimates (cf. \cite{Casas1997,Gong-2013}):
		\begin{eqnarray}
			\|z\|_{L^\infty(I;L^2(\Omega))}\leq C\|g\|_{L^1(I;L^2(\Omega))},\quad 
			\|z\|_{L^\infty(I;L^2(\Omega))}\leq C\|g\|_{L^2(I;H^{-1}(\Omega))},\nonumber
		\end{eqnarray}
		we conclude that $\{u_n \}_n$ is bounded in the space $L^\infty(I;L^2(\Omega))$ by setting $g:=\psi_0\in \mathcal{D}(I\times \Omega)$ and using the density of $\mathcal{D}(I\times \Omega)$ in $L^1(I;L^2(\Omega))$, and also bounded in $L^2(I;H_0^1(\Omega))$ by setting $g:=\psi_0-\frac{\partial \psi_j}{\partial x_j}$, $\psi_j\in \mathcal{D}(I\times \Omega)$, $j=1,\dots,d$ and using the density of $\mathcal{D}(I\times \Omega)$ in $L^2(I;H^{-1}(\Omega))$ (cf. \cite{CasasClasonKunisch-2013,CasasKunisch-2016}), respectively. Thus, we can extract a subsequence, still denoted by $\{u_n \}_n$, such that $u_n\rightarrow u$ weakly in $L^2(I;H_0^1(\Omega))$ and $ L^\infty(I;L^2(\Omega))$. 
		
		For any $g\in L^2(I;L^2(\Omega))$, let $z_g\in H^1(I;L^2(\Omega))\cap L^2(I;H^2(\Omega)\cap H^1_0(\Omega))$ be the solution of equation \eqref{backward_PDE} with $z_T=0$. Multiplying by $z_g$ in both sides of equation (\ref{forward_PDE_state}) and integrating by parts give
		\begin{eqnarray}
			(u_n,g)_{L^2(I;L^2(\Omega))}=\int_{I_c}\int_\omega z_g(x,t)dxdq_n(t),\nonumber
		\end{eqnarray} 
		which yields the identity (\ref{identity_1}) with $f=0,\,u_0=0$ by passing to the limit in the above identity. Therefore, $u$ is the very weak solution of equation \eqref{PDE_state}. By the weak lower semicontinuity of $\Vert\cdot\Vert_{L^2(I;H^1_0(\Omega))}$ and $\Vert\cdot\Vert_{L^\infty(I;L^2(\Omega))}$,  we can obtain the estimate (\ref{state_stability}) of $u$.
		
		Since $I_c=(t_1,t_2)\subseteq (0,T)$, there exist $\hat{t}$ and $\tilde{t}$ satisfying $t_2<\hat{t}<\tilde{t}<T$, such that $(\tilde{t},T)\subseteq (\hat{t},T)\subseteq (t_2,T)$. Therefore, we consider a smooth cut-off function $\tilde{\omega}$ with the following properties:
		\begin{align*} 
				\tilde{\omega}(t)\in [0,1]\quad \forall t\in [0,T];\quad
				\tilde{\omega}(t)=1\quad\forall t\in (\tilde{t},T);\quad
				\tilde{\omega}(t)=0\quad\forall t\in (0,\hat{t}].
		\end{align*}
		Let $\tilde{u}:=\tilde{\omega}u$. Since $\bar{I}_c\cap {\rm supp}\,\tilde{u}=\emptyset $, $\tilde{u}$ satisfies the following equation:
		\begin{align}
			\left\{
			\begin{aligned}
				\partial_t \tilde{u} - \Delta \tilde{u} & =F
				&&\mbox{in}\,\,\,\Omega\times(\hat{t},T),\\
				\tilde{u}&=0 && \mbox{on}\,\, \Gamma\times (\hat{t},T),\\
				\tilde{u}(\hat{t})&=0 && \mbox{in}\,\,\Omega,
			\end{aligned}
			\right.
		\end{align}
		where $F:=\partial_t\tilde{\omega}u+\tilde{\omega}f$. Since $F\in L^2((\hat{t},T);L^2(\Omega))$, we can obtain $\tilde{u}\in L^2((\hat{t},T);H^2(\Omega)\cap H_0^1(\Omega))\cap H^1((\hat{t},T);L^2(\Omega))$ and there holds the following estimate:
		\begin{equation}
			\begin{split}
				\Vert \tilde{u}\Vert_{L^2((\hat{t},T);H^2(\Omega)\cap H_0^1(\Omega))}+\Vert \tilde{u}\Vert_{H^1((\hat{t},T);L^2(\Omega))}&\le C\Vert F\Vert_{L^2((\hat{t},T);L^2(\Omega))} \\
				&\le C(\Vert f\Vert_{L^2((\hat{t},T);L^2(\Omega))}+\Vert u\Vert_{L^2((\hat{t},T);L^2(\Omega))})\\\nonumber
				&\le C(\Vert f\Vert_{L^2(I;L^2(\Omega))}+\Vert q\Vert_{\mathcal{M}(\bar{I}_c;L^2(\omega))}+\Vert u_0\Vert_{L^2(\Omega)}),
			\end{split}
		\end{equation}
		where we have used the estimate \eqref{state_stability}. From the above inequality we obtain 
		\begin{eqnarray*}
				\Vert {u}\Vert_{L^2((\tilde{t},T);H^2(\Omega)\cap H_0^1(\Omega))}+\Vert {u}\Vert_{H^1((\tilde{t},T);L^2(\Omega))}&=& \Vert \tilde{u}\Vert_{L^2((\tilde{t},T);H^2(\Omega)\cap H_0^1(\Omega))}+\Vert \tilde{u}\Vert_{H^1((\tilde{t},T);L^2(\Omega))} \\
				&\le& \Vert \tilde{u}\Vert_{L^2((\hat{t},T);H^2(\Omega)\cap H_0^1(\Omega))}+\Vert \tilde{u}\Vert_{H^1((\hat{t},T);L^2(\Omega))}\\\nonumber
				&\le& C(\Vert f\Vert_{L^2(I;L^2(\Omega))}+\Vert q\Vert_{\mathcal{M}(\bar{I}_c;L^2(\omega))}+\Vert u_0\Vert_{L^2(\Omega)}).
		\end{eqnarray*}
		Therefore, we complete the proof of the estimate \eqref{local_estimate}.
	\end{proof}

	With the help of Theorem \ref{Thm:existence_state}, the identity \eqref{identity_1} in Definition \ref{def_1} is equivalent to the following one
	\begin{align}\label{identity_2}
		(u,g)_*+(u(T),z_T)_{L^2(\Omega)}=\int_I(f,z)_{L^2(\Omega)}dt+\langle q,z\rangle_{\bar I_c\times \omega}+(u_0,z(0))_{L^2(\Omega)}
	\end{align}
	for any $(g,z_T)\in S\times L^2(\Omega)$, where $(u,g)_*:=(u,g)_{L^2(I;L^2(\Omega))}$ for $S:= L^2(I;L^2(\Omega))$ and  $(u,g)_*:=\langle g,u\rangle_{L^2(I;H^{-1}(\Omega)),L^2(I;H_0^1(\Omega))}$ defined by 
	\begin{equation}\label{identity_3}
		\langle g,u\rangle_{L^2(I;H^{-1}(\Omega)),L^2(I;H_0^1(\Omega))}:=\int_I\langle u,-\partial_t z\rangle_{H^1(\Omega),H^{-1}(\Omega)}+(\nabla u,\nabla z)_{L^2(\Omega),L^2(\Omega)}dt\end{equation}
	for $S:= L^2(I;H^{-1}(\Omega))$, where $z\in C(\bar{I}_c;L^2(\omega))$ satisfies (\ref{backward_PDE})  with the right-hand side $g$ and $z_T\in L^2(\Omega)$.
	
	
	In fact, taking $\hat{z}=z-\tilde{z}$, where $\tilde{z}$ is the solution of equation (\ref{backward_PDE}) with $g=0$ and initial data $\tilde{z}(T)=z(T)=z_T$, then $\hat{z}$ satisfies (\ref{backward_PDE}) with the right-hand side $g$ and initial data $z_T=0$. In other words, $\hat{z}$ can be chosen as a test function in  Definition \ref{def_1}, i.e, 
	\begin{align*}
		(u,g)_{L^2(I;L^2(\Omega))}=&\int_I(f,\hat{z})_{L^2(\Omega)}dt+\langle q,\hat{z}\rangle_{\bar I_c\times \omega}+(u_0,\hat{z}(0))_{L^2(\Omega)}\\
		=&\int_I(f,{z})_{L^2(\Omega)}dt+\langle q,{z}\rangle_{\bar I_c\times \omega}+(u_0,{z}(0))_{L^2(\Omega)}\\
		&-\left(\int_I(f,\tilde{z})_{L^2(\Omega)}dt+\langle q,\tilde{z}\rangle_{\bar I_c\times \omega}+(u_0,\tilde{z}(0))_{L^2(\Omega)}\right)\\
		=&\int_I(f,{z})_{L^2(\Omega)}dt+\langle q,{z}\rangle_{\bar I_c\times \omega}+(u_0,{z}(0))_{L^2(\Omega)}-\mathscr{L}(z_T),
	\end{align*}
	where $\mathscr{L}(z_T):=\int_I(f,\tilde{z})_{L^2(\Omega)}dt+\langle q,\tilde{z}\rangle_{\bar I_c\times \omega}+(u_0,\tilde{z}(0))_{L^2(\Omega)}$. It  is easy to check that $\mathscr{L}$ is a bounded linear functional of $z_T\in L^2(\Omega)$. Therefore, there exists a unique $\theta\in L^2(\Omega)$ such that $\mathscr{L}(z_T)=(\theta,z_T)_{L^2(\Omega)}$ by the Riesz representation theorem. Obviously, $\theta=u(T)$. Then, the identity \eqref{identity_2} holds.
	
	In order to show that the optimal control problem \eqref{min-J} has a unique solution, we have to provide a continuity property of the control-to-observation mapping under the weak-* topology.	
	\begin{proposition}\label{Thm:continuity}
		Let $\{q_n \}_{n\in\mathbb{N}_+}\subset \mathcal{M}(\bar I_c;L^2(\omega))$ be a sequence of control variables such that $q_n\stackrel{*}{\rightharpoonup} q$ in $\mathcal{M}(\bar I_c;L^2(\omega))$. Assume that $u_n:=u(q_n)$ and $u:=u(q)$ are the corresponding solutions to the state equation (\ref{PDE_state}) associated with $q_n$ and $q$, respectively. Then we have
		\begin{align*}
			\|u_n-u\|_{L^2(I;L^2(\Omega))}\rightarrow 0\ \ \mbox{\rm and}\ \  \|u_n(\cdot,T)-u(\cdot,T)\|_{H^1(\Omega)}\rightarrow 0\quad \mbox{\rm for}\ n\to\infty.
		\end{align*}		
	\end{proposition}
	\begin{proof}
		The main idea of the proof is to apply the definition of very weak solutions to $u-u_n$. To this end, we note that $u-u_n$ satisfies the equation \eqref{PDE_state} with $f=0$, $u_0=0$, and $q$ replaced by $q-q_n$. Therefore, taking any $ g\in L^2(I;H^{-1}(\Omega))$ and $z_T=0$, let $z\in C(\bar{I};L^2(\Omega))$ be the solution of equation \eqref{backward_PDE}. Using (\ref{identity_2})  there holds
		\begin{eqnarray}
			\langle g, u-u_n\rangle_{L^2(I;H^{-1}(\Omega)),L^2(I;H_0^{1}(\Omega))}=\langle q-q_n, z\rangle_{\bar I_c\times \omega}
			\to 0\quad \mbox{\rm as}\ \ n\to\infty,
		\end{eqnarray}
		which implies that $u_n\rightharpoonup u$ in $L^2(I;H^1_0(\Omega))$. Furthermore, we obtain $\|u-u_n\|_{L^2( I;L^2(\Omega))}\rightarrow 0$  by the compact embedding $L^2(I;H^1_0(\Omega))\hookrightarrow L^2( I;L^2(\Omega))$. This proves the first statement. 
		
		Below we will verify the second statement. Since $u-u_n$ satisfies the equation \eqref{PDE_state} with $f,\,u_0$ and $q$ replaced by $0$, $0$, $q-q_n$, respectively, then applying the estimate \eqref{local_estimate} in Theorem \ref{Thm:existence_state} to $u-u_n$ yields the following estimate:
		\begin{eqnarray*}
			\|(u-u_n)(T)\|_{H^1(\Omega)}\le C\|u-u_n\|_{L^2((\hat{t},T);L^2(\Omega))}
			\le C\|u-u_n\|_{L^2(I;L^2(\Omega))}
			\to 0\quad \mbox{\rm as}\ \ n\to \infty,
		\end{eqnarray*}
		where we have used the first statement.
		This finishes the proof.
	\end{proof}
	
	\section{Optimal control problems}
	With the above preparations, we are in the position to study the existence and uniqueness of solutions to  the optimal control problem (\ref{min-J})-(\ref{PDE_state}), and derive the first order optimality system and regularity results of the solution.
	\subsection{Well-posedness of the optimal control problem}\label{subsec_1}
	Recall that $X:=L^2(I;L^2(\Omega))\times L^2(\Omega)$ is the observation space, then we introduce the control-to-observation operator $S:\mathcal{M}(\bar{I}_c;L^2(\omega))\to X$ as
	$$ Sq:=(S_1q,S_2q), $$
	where $S_1q:=u_q$ and $S_2q:=u_q(T)$, and $u_q$ solves equation \eqref{PDE_state} with the control variable $q$ on the right-hand side. Theorem \ref{Thm:existence_state} and Proposition \ref{Thm:continuity} imply that the operator $S$ is well-defined, affine linear and bounded, and weak continuous under the weak-$^*$ topology in $\mathcal{M}(\bar{I}_c;L^2(\omega))$.  Obviously, the operator $S$ is injective since $I_c\subseteq I$. With the help of the control-to-observation operator $S$ the reduced cost functional of \eqref{min-J} can be defined as
	\begin{equation}\label{reduced_j}
		j(q):=J_1(q)+J_2(q)\quad \forall q\in \mathcal{M}(\bar{I}_c;L^2(\omega)),
	\end{equation}
	where $$J_1(q):=\frac{1}{2}\|S_1q-u_d\|^2_{L^2(I;L^2(\Omega))}+\frac{\beta}{2}\|S_2q-u_T\|^2_{L^2(\Omega)},\quad\ J_2(q):=\alpha\|q\|_{\mathcal M(\bar{I}_c;L^2(\omega))}.$$
	$J_1(q)$ is a quadratic functional of tracking type, which is continuous under the weak-$^*$ topology in $\mathcal M(\bar{I}_c;L^2(\omega))$ and strictly convex by the weak-$^*$ continuity and injection of the operator $S$. On the other hand,  $J_2(q)$ is weakly-$^*$ lower semicontinuous in space $\mathcal M(\bar{I}_c;L^2(\omega))$. Therefore, we conclude that the reduced functional $j$ is also weakly-$^*$ lower semicontinuous and strictly convex. With this observation we can provide the following result.
	\begin{theorem}\label{Thm:existence_OCP}
		The optimal control problem (\ref{min-J})-(\ref{PDE_state}) admits a unique solution $(\bar{u},\bar{q})\in X\times \mathcal{M}(\bar I_c;L^2(\omega))$, where $\bar{q}$ is an optimal control that minimizes the reduced cost functional $\eqref{reduced_j}$ and $\bar{u}$ is the optimal state that solves the state equation \eqref{PDE_state} associated with $\bar{q}$.
	\end{theorem}
	\begin{proof}
		According to Theorem \ref{Thm:existence_state}, the objective functional $j$ is well defined on $\mathcal{M}(\bar I_c;L^2(\omega))$. For the existence of solutions, we follow the standard arguments. Since $j\ge 0$ is bounded from below on $\mathcal{M}(\bar I_c;L^2(\omega))$, we can find a minimizing sequence $\{q_n\}$ with $$\lim_{n\to \infty}j(q_n)=\mathop{\rm inf}_{q\in \mathcal{M}(\bar I_c;L^2(\omega))}j(q)=j^*\quad {\rm and}\quad \|q_n\|_{\mathcal{M}(\bar I_c;L^2(\omega))}\le\frac{1}{\alpha}j(q_n)\le C.$$
		On the other hand, the predual space ${C}(\bar I_c;L^2(\omega))$ is separable, then the bounded set in  $\mathcal{M}(\bar I_c;L^2(\omega))$ is weakly-$^*$ compact by the Banach-Alaoglu theorem. Hence, we can extract a weakly-$^*$ convergent subsequence, still denoted by $\{q_n\}$, such that $q_n\stackrel{*}{\rightharpoonup} \bar{q}$ in $\mathcal{M}(\bar I_c;L^2(\omega))$ for some $\bar{q}\in \mathcal{M}(\bar I_c;L^2(\omega))$. Let $u_n$ and  $\bar u$ be the state corresponding to $q_n$ and $\bar q$, respectively. Then $u_n\to \bar{u}$ in $X$  since the operator $S$ is weakly-$^*$ continuous. It is easy to check that $(\bar{u},\bar{q})$ is an optimal pair. In fact, $j$ is weakly-$^*$ lower semicontinuous, then 
		$$j(\bar{q})\le \mathop{\lim\inf}_{n\to \infty}\ j(q_n)=j^*,$$
		which means that $\bar{q}$ is optimal, i.e., $(\bar{u},\bar{q})$ is an optimal pair.
		
		Furthermore, the control-to-observation mapping $S$ is injective, thus the objective functional $j$ is strictly convex. Therefore, the optimal pair $(\bar{u},\bar{q})$ is unique.
	\end{proof}
	\begin{remark}
		As pointed out in \cite{KunischPieperVexler-2014}, if the state observation is of the form $\chi_{I_o\times\Omega_o}(u-u_d)\in L^2(I_o;L^2(\Omega_o))$ with ${\rm dist}(I_o,I_c)>0$ and $\beta=0$ where $I_o\subset I$ is an observation time window. The objective functional is no longer strictly convex since the control-to-observation operator is not injective, and thus the optimal pair of the optimization problem (\ref{min-J})-(\ref{PDE_state}) is not unique. However, in the current paper we do not consider this case and focus only on the situation of $I_o=I$, i.e., $I_c\subset I_o$.
	\end{remark}

	
	\subsection{First order optimality system}
	Below, we are in the position to derive the first order optimality condition.
	\begin{theorem}\label{Theorem_optimal_system}
		A control $\bar{q}\in \mathcal M(\bar{I}_c;L^2(\omega))$ and an associated state $\bar{u}\in  L^2(I;H^1_0(\Omega))\cap L^\infty(I;L^2(\Omega))$ are an optimal pair of the optimal control problem (\ref{min-J})-(\ref{PDE_state}), if and only if there exists an adjoint state $\bar{\varphi}\in L^2(I;H_0^1(\Omega))\cap H^1(I;L^2(\Omega))\hookrightarrow C(\bar I;L^2(\Omega))$ satisfying
		\begin{align}\label{adjoint_equation}
			\left\{
			\begin{aligned}
				-\partial_t \bar{\varphi} - \Delta \bar{\varphi} & =\bar{u}-u_d
				&&\mbox{in}\,\,\,\Omega\times(0,T),\\
				\bar{\varphi}&=0 && \mbox{on}\,\, \Gamma\times (0,T),\\
				\bar{\varphi}(T)&=\beta(\bar{u}(T)-u_T) && \mbox{in}\,\,\Omega ,
			\end{aligned}
			\right.
		\end{align}
		where $u_T\in L^2(\Omega)$, $u_d\in L^2(I;L^2(\Omega))$, such that the following subgradient condition holds:
		\begin{equation}\label{optimal_condition_1}
			0\in\bar{\varphi}|_{\bar I_c\times \omega}+\alpha\mathscr{\partial} \|\cdot\|_{\mathcal {M}(\bar{I}_c;L^2(\omega))}(\bar{q})\quad\mbox{\rm in}\ (\mathcal M(\bar{I}_c;L^2(\omega)))^*
		\end{equation}
		i.e.,
		\begin{equation}\label{optimal_condition_2}
			-\langle p-\bar{q},\bar{\varphi}\rangle_{\bar I_c\times \omega}+\alpha \|\bar{q}\|_{\mathcal M(\bar{I}_c;L^2(\omega))}\le \alpha \|p\|_{\mathcal M(\bar{I}_c;L^2(\omega))}\quad\forall p\in \mathcal M(\bar{I}_c;L^2(\omega)),
		\end{equation}
		where $\mathscr{\partial} \|\cdot\|_{\mathcal M(\bar{I}_c;L^2(\omega))}(\bar{q})$ denotes the set of subgradients of  $\|\cdot\|_{\mathcal M(\bar{I}_c;L^2(\omega))}$ at $\bar{q}$, which is nonempty since  $\|\cdot\|_{\mathcal M(\bar{I}_c;L^2(\omega))}$ is a convex functional on $\mathcal M(\bar{I}_c;L^2(\omega))$.
		
		Furthermore, from the condition \eqref{optimal_condition_2} we can easily conclude the following relation between the optimal control $\bar{q}$ and the adjoint state $\bar{\varphi}$:
		\begin{align}
			\alpha \|\bar{q}\|_{\mathcal{M}(\bar I_c;L^2(\omega))}+\langle \bar{q},\bar{\varphi}\rangle_{\bar I_c\times \omega} &=0 ,\label{OCP_OPT}\\
			\|\bar{\varphi}\|_{{C}(\bar I_c;L^2(\omega))}\left\{
			\begin{aligned}\label{adjoint_prop}
				&=\alpha && \mbox{if}\,\, \bar{q}\neq 0,\\
				&\leq \alpha && \mbox{if}\,\, \bar{q}=0.
			\end{aligned}
			\right.
		\end{align}
	\end{theorem} 
	\begin{proof}
		We split the reduced cost functional $j$ into the sum of two parts in \eqref{reduced_j}, where $J_1$ is differentiable and $J_2$ is subdifferentiable. We use $J_1^\prime(\bar{q})$ and $\mathscr{\partial}J_2(\bar{q})$ to denote the Fr\'echet derivative of $J_1$ at $\bar{q}$ and subgradients of $J_2$ at $\bar{q}$, respectively. By the calculus rules of subdifferentials for convex functions, there holds (see, e.g.,~\cite[Section 5.3]{IER1999}):
		\begin{equation}\label{condition_2}
			j(\bar{q})=\!\!\!\mathop{\rm min}_{p\in \mathcal M(\bar{I}_c;L^2(\omega))}j(p)\quad \mbox{\rm if and only if} \quad 0\in J_1^\prime(\bar{q})+\mathscr{\partial}J_2(\bar{q}),
		\end{equation} 
		where for any $p\in\mathcal M(\bar{I}_c;L^2(\omega))$, $J_1^\prime(\bar{q})$ and $\mathscr{\partial}J_2(\bar{q})$ satisfy
		\begin{equation}\label{derivative}
			\begin{split}
				J_1^\prime(\bar{q})(p-\bar{q})=&(\bar{u}-u_d,u_{p}-\bar u)_{L^2(I;L^2(\Omega))}+\beta(\bar{u}(T)-u_T,u_{p}(T)-\bar u(T)),\\ 
				\forall\xi\in\mathscr{\partial}J_2(\bar{q}),\quad& \langle\xi,p-\bar{q}\rangle_{(\mathcal M(\bar{I}_c;L^2(\omega)))^*,\mathcal M(\bar{I}_c;L^2(\omega))}\le J_2({p})-J_2(\bar{q}),			
			\end{split}
		\end{equation}
		respectively, where $(\mathcal M(\bar{I}_c;L^2(\omega)))^*$ denotes the topological dual of the space $\mathcal{M}(\bar{I}_c;L^2(\omega))$ and $\langle\cdot,\cdot\rangle_{(\mathcal M(\bar{I}_c;L^2(\omega)))^*,\mathcal{M}(\bar{I}_c;L^2(\omega))}$ denotes the duality pairing between $(\mathcal M(\bar{I}_c;L^2(\omega)))^*$ and $M(\bar{I}_c;L^2(\omega))$, $u_p$ is the solution of problem (\ref{PDE_state}) with $q$ replaced by $p$. In order to give an explicit  representation of $J_1^\prime(\bar{q})(p-\bar{q})$ with respect to $p-\bar{q}$, let $\bar{\varphi}$ be the solution of equation \eqref{backward_PDE} with $g=\bar{u}-u_d,\ z_T=\beta(\bar{u}(T)-u_T)$,  and then apply the identity \eqref{identity_2} to the difference $u_p-\bar{u}$ to deduce
		\begin{align}
			(\bar{u}-u_d,u_{p}-\bar{u})_{L^2(I;L^2(\Omega))}+\beta(\bar{u}(T)-u_T,u_{p}(T)-\bar{u}(T))=\langle p-\bar{q},\bar{\varphi}\rangle_{\bar I_c\times \omega}\label{identity_4}
		\end{align}
		for any  $p\in \mathcal M(\bar{I}_c;L^2(\omega))$. Furthermore, we obtain
		$$J_1^\prime(\bar{q})(p)=\langle p,\bar{\varphi}\rangle_{\bar I_c\times \omega}\quad \forall p\in \mathcal M(\bar{I}_c;L^2(\omega)),$$
		which means that $J_1^\prime(\bar{q})=\bar{\varphi}|_{\bar I_c\times \omega}$. Therefore, combining with \eqref{condition_2} we deduce the optimality condition  \eqref{optimal_condition_1} which claims that $-
		\bar{\varphi}|_{\bar I_c\times \omega}\in\alpha\mathscr{\partial}\|\cdot\|_{\mathcal{M}(\bar{I}_c;L^2(\omega))}$, i.e., \eqref{optimal_condition_2}.

		Testing \eqref{optimal_condition_2} with $p=2\bar{q}$ and $p=0$ we arrive at (\ref{OCP_OPT}). Furthermore, it follows from setting $p=\bar{q}-r$ in \eqref{optimal_condition_2} for arbitrary $r\in \mathcal M(\bar{I}_c;L^2(\omega))$ that
		\begin{eqnarray}
			\langle r,\varphi \rangle_{\bar I_c\times \omega}\leq J_2(\bar{q}-r)-J_2(\bar{q})\leq J_2(r)=\alpha\|r\|_{\mathcal{M}(\bar I_c;L^2(\omega))}\quad \forall r\in \mathcal{M}(\bar I_c;L^2(\omega)).\nonumber
		\end{eqnarray}
		Hence, we obtain
		\begin{eqnarray}
			{\|\varphi\|_{{C}(\bar I_c;L^2(\omega))}=\sup\limits_{\|r\|_{\mathcal{M}(\bar I_c;L^2(\omega))}\leq 1}\langle r,\varphi \rangle_{\bar I_c\times \omega}\leq \alpha},
		\end{eqnarray}
		this verifies (\ref{adjoint_prop}) in view of (\ref{OCP_OPT}).
	\end{proof}
	
	In the following we will derive the sparsity structure in time of $\bar q$.
	\begin{theorem}\label{Thm:Jordan}
		Let $\bar{q}$ be the optimal control of the optimization problem (\ref{min-J})-(\ref{PDE_state}) and $\bar{\varphi}$ be the optimal adjoint state defined by equation (\ref{adjoint_equation}), then there holds
		\begin{eqnarray}
			{\rm supp}|\bar{q}|\subset \{t\in \bar I_c: \|\bar{\varphi}(t)\|_{L^2(\omega)}=\alpha\},\label{Jordan_1}\\
			\bar{q}'(t,x)=-{1\over\alpha}\bar{\varphi}(t,x)\quad\mbox{in}\ L^1(I_c,|\bar{q}|;L^2(\omega)),\label{Jordan_2}
		\end{eqnarray}
		where $d\bar{q}=\bar{q}'d|\bar{q}|$ denotes the polar decomposition of $\bar{q}$.
	\end{theorem}
	\begin{proof}
		The idea of proof  follows from \cite[Theorem 2.12]{KunischPieperVexler-2014}, see also \cite[Theorem 3.3]{CasasClasonKunisch-2013}. Here we sketch it for completeness. Applying the polar decomposition of $q$ in (\ref{OCP_OPT}) we have	
		\begin{eqnarray}\label{integration_1}
			\int_{\bar I_c}(\alpha +(\bar{q}'(t),\bar{\varphi})_{L^2(\omega)})d|\bar{q}|(t)=0.\label{support_proof_2}
		\end{eqnarray}
		On the other hand, it follows from (\ref{Polar}) and (\ref{adjoint_prop}) that
		\begin{eqnarray}
			(\bar{q}'(t),\bar{\varphi})_{L^2(\omega)}\geq -\|\bar{q}'(t)\|_{L^2(\omega)}\|\bar{\varphi}\|_{L^2(\omega)}\geq -\alpha\quad \ |\bar{q}|-\mbox{a.e.}\ t\in \bar I_c,\label{support_proof_1}
		\end{eqnarray}
		i.e., the integrand in \eqref{integration_1} is nonnegative.  Thus, it must be zero $\ |\bar{q}|$- almost everywhere, that is, 
		\begin{eqnarray}
			-(\bar{q}'(t),\bar{\varphi})_{L^2(\omega)}=\alpha\quad \mbox{for}\ |\bar{q}|-\mbox{almost\  all}\ t\in \bar I_c.
		\end{eqnarray}
		Therefore, in view of \eqref{support_proof_1} we have the identity: 
		\begin{eqnarray}
			(\bar{q}'(t),\bar{\varphi})_{L^2(\omega)}= -\|\bar{q}'(t)\|_{L^2(\omega)}\|\bar{\varphi}\|_{L^2(\omega)}= -\alpha,\quad \mbox{for}\ |\bar{q}|-\mbox{almost\  all}\ t\in \bar I_c,
		\end{eqnarray}
		which is equivalent to
		\begin{eqnarray}
			\|\bar{\varphi}(t)\|_{L^2(\omega)}=\alpha\quad\mbox{and}\quad \bar{\varphi}(t,x)=-\alpha \bar{q}'(t,x)
		\end{eqnarray}
		for $|\bar{q}|$-almost all $t\in \bar I_c$ and a.e. $x\in \omega$. Thus, we finish the proof of (\ref{Jordan_2}).  
		
		In view of (\ref{Polar}), we have 
		$$\|\bar{\varphi}(t)\|_{L^2(\omega)}=\alpha\|\bar{q}^\prime(t)\|_{L^2(\omega)} = \alpha \quad \mbox{for}\ |\bar{q}|\ \mbox{-almost all}\ t\in \bar I_c.$$
		Namely,
		$$|\bar{q}|(\bar{I}_c)=|\bar{q}|(\mbox{supp}\ |\bar{q}|)=|\bar{q}|(\{t\in\bar{I}_c:\|\bar{\varphi}(t)\|_{L^2(\omega)}=\alpha\}),$$
		which means that $\mbox{supp}\ |\bar{q}|\subseteq \{t\in\bar{I}_c:\|\bar{\varphi}(t)\|_{L^2(\omega)}=\alpha\}$. This finishes the proof.
	\end{proof}
	
	In view of \eqref{Jordan_1} in Theorem \ref{Thm:Jordan}, we find that the optimal control $\bar q\in \mathcal{M}(\bar I_c;L^2(\omega))$ has sparsity pattern in time that is independent of the spatial domain. Moreover, if  $\|\bar\varphi(t)\|_{L^2(\omega)}=\alpha$ holds for a finite set of time instances, namely, $\{t\in \bar I_c: \|\varphi(t)\|_{L^2(\omega)}=\alpha\}=\{\tau_i\}_{i=1}^N$, then $\bar q$ has the representation $\bar q(t,x)=\sum\limits_{i=1}^N\chi_\omega \bar q_i(x)\delta_{\tau_i}$ such that $\bar q_i\in L^2(\omega)$ (cf. \cite{KunischPieperVexler-2014}). This is exactly the impulse control problem, studied extensively in the literature, e.g., \cite{DuanWang-2020,DuanWangZhang-2019,Phung-Wang-Xu-2017,Qin-Wang-2017,Trelat-Wang-Zhang-2016}, and the references therein. 

	\begin{proposition}
		There exists $\alpha_0>0$ such that the optimal control $\bar{q}=0$ when $\alpha>\alpha_0$.
	\end{proposition}
	\begin{proof}
		The main idea follows from \cite[Corollary 3.5]{CasasClasonKunisch-2013}, see also \cite[Proposition 2.2]{CasasClasonKunisch-2012}, and we sketch it here. Note that 
		\begin{align}
			\frac{1}{2}\|\bar{u}-u_d\|_{L^2(I;L^2(\Omega))}^2
			+\frac{\beta}{2}\|\bar{u}(T)-u_T\|_{L^2(\Omega)}^2\leq J(\bar{q})\leq J(0),\nonumber
		\end{align}
		where $J(0)$ is independent of $\alpha$. Then we have obtained a uniform upper bound of $\|\bar{u}-u_d\|_{L^2(I;L^2(\Omega))}$ and $\|\bar{u}(T)-u_T\|_{L^2(\Omega)}$ with respect to $\alpha$. For any $u_d\in L^2(I;L^2(\Omega)),\ u_T\in L^2(\Omega)$, let $\bar{\varphi}\in H^1(I;H^{-1}(\Omega))\cap L^2(I;H^1_0(\Omega))\hookrightarrow C(\bar I;L^2(\Omega))$ be the optimal adjoint state defined by equation \eqref{adjoint_equation} with the following estimate:
		\begin{align*}
			\|\bar{\varphi}(t)\|_{L^2(\Omega)}\leq C(\|\bar{u}-u_d\|_{L^2(I;L^2(\Omega))}+\beta\|\bar{u}(T)-u_T\|_{L^2(\Omega)})\leq 2CJ(0),\nonumber
		\end{align*}
		where we have used Lemma \ref{stateEexist1}. Setting $\alpha_0=2CJ(0)$, it follows from (\ref{Jordan_1}) in Theorem \ref{Thm:Jordan} that $\bar{q}=0$ for all $\alpha>\alpha_0$. This finishes the proof.
	\end{proof}
	
	\subsection{The regularity of solutions}  In this subsection we prepare to state the  regularity of solutions to the optimality system, which will be used in the finite element approximation to the optimal state and adjoint state.
	
	\begin{theorem}\label{Theorem3-5}
		For any $u_d,\, f\in L^2(I;L^2(\Omega)),\,u_0,\,u_T\in L^2(\Omega)$, let $(\bar{q},\bar{u},\bar{\varphi})$ be the optimal solution of the optimal control problem (\ref{min-J})-(\ref{PDE_state}), where $\bar{q},\bar{u}$ and  $\bar{\varphi}$ are the optimal control, optimal state and adjoint state, respectively. Then there hold
		\begin{align*}
			&\bar{u}\in L^2(I;H^1_0(\Omega))\cap L^\infty(I;L^2(\Omega)),\quad \bar{\varphi}\in H^1(I;H^{-1}(\Omega))\cap L^2(I;H^1_0(\Omega)),\quad \bar{q}\in \mathcal{M}(\bar{I}_c;H^1(\omega)),\\
			&\bar{u}|_{(\tilde{t},T)}\in L^2((\tilde{t},T);H^2(\Omega)\cap H^1_0(\Omega))\cap H^1((\tilde{t},T);L^2(\Omega))\hookrightarrow C([\tilde{t},T];H^1_0({\Omega})),\\
			&\bar{\varphi}\in L^2(I_c;H^2(\Omega)\cap H^1_0(\Omega))\cap H^1(I_c;L^2(\Omega))\hookrightarrow C(\bar{I}_c;H^1_0({\Omega})),
		\end{align*}
		where $(\tilde{t},T)\cap I_c=\emptyset$, and there hold the following stability estimates 
		\begin{align*}
			&\|\bar{u}\|_{L^2(I;H^{1}_0(\Omega))}+\|\bar{u}\|_{L^\infty(I;L^2(\Omega))}\\
			+&\|\bar{\varphi}\|_{H^1(I;H^{-1}(\Omega))}+\|\bar{\varphi}\|_{L^2(I;H^1_0(\Omega))}+\|\bar{\varphi}\|_{C(\bar{I};L^2(\Omega))}+\|\bar{q}\|^\frac{1}{2}_{ \mathcal{M}(\bar{I}_c;H^1(\omega))}\\
			+&\|\bar{u}\|_{H^1((\tilde{t},T);L^2(\Omega))}+\|\bar{u}\|_{L^2((\tilde{t},T);H^2(\Omega)\cap H^1_0(\Omega))}+\|\bar{u}\|_{C([\tilde{t},T];H^1_0(\Omega))}\\
			\le& C\left(\|f\|_{L^2(I;L^2(\Omega))}+\|u_d\|_{L^2(I;L^2(\Omega))}+\|u_T\|_{L^2(\Omega)}+\|u_0\|_{L^2(\Omega)}\right).
		\end{align*}
		Moreover, if $u_T\in H^1_0(\Omega)$, then the adjoint state has the following improved regularity:
		\begin{equation}
			\bar{\varphi}\in H^1(I;L^2(\Omega))\cap L^2(I;H^2(\Omega)\cap H^1_0(\Omega))\nonumber
		\end{equation}
		with the estimate
		\begin{eqnarray*}
			&&\|\bar{\varphi}\|_{H^1(I;L^2(\Omega))}+\|\bar{\varphi}\|_{L^2(I;H^2(\Omega)\cap H^1_0(\Omega))}+\|\bar{\varphi}\|_{C(\bar{I};H^1_0(\Omega))}\\
			&\le& C\left(\|f\|_{L^2(I;L^2(\Omega))}+\|u_d\|_{L^2(I;L^2(\Omega))}+\|u_T\|_{H^1(\Omega)}+\|u_0\|_{L^2(\Omega)}\right).
		\end{eqnarray*}
	\end{theorem}
	\begin{proof}
		To begin with, we first consider the regularity of  $\bar{u},\, \bar{\varphi}$ and derive the associated estimates. By using Theorem \ref{Thm:existence_state} and  Lemma \ref{stateEexist1}, $\bar{u}$ and $\bar{\varphi}$ have the above mentioned regularity, but the local regularity of $\bar{\varphi}$ and the associated stability estimates depend on $\|\bar{q}\|_ {\mathcal{M}(\bar{I}_c;L^2(\omega))}$. Therefore, we first provide the estimate for $\|\bar{q}\|_ {\mathcal{M}(\bar{I}_c;L^2(\omega))}$.
		
		Let $\tilde{u}_0$ be the solution of equation \eqref{PDE_state} with $q=0$. Using the identity \eqref{OCP_OPT} there holds
		\begin{align*}
			\alpha\|\bar{q}\|_{\mathcal{M}(\bar{I}_c;L^2(\omega))}
			\le& J(\bar{u},\bar{q})
			\le J(\tilde u_0,0)\\
			\le&C(\|u_d\|^2_{L^2(I;L^2(\Omega))}+\|f\|^2_{L^2(I;L^2(\Omega))}+\|u_0\|^2_{L^2(\Omega)}+{\beta}\|u_T\|^2_{L^2(\Omega)}).
		\end{align*}	
		
		Below, we show that $\bar{\varphi}\in H^1(I_c; L^2(\Omega))\cap L^2(I_c;H^2(\Omega)\cap H^1_0(\Omega))\hookrightarrow C(\bar{I}_c;H^1_0(\Omega))$ and $\bar{q}\in \mathcal M(\bar{I}_c;H^1(\omega))$. 
		Since $I_c=(t_1,t_2)\subseteq I$ is relatively compact, there exist $\tilde{t}_1$, $\tilde{t}_2$ satisfying $\tilde{t}_1<t_1<t_2<\tilde{t}_2$, such that $I_c\subseteq (\tilde{t}_1,\tilde{t}_2)$. Let $\tilde{\omega}$ be a smooth cut-off function satisfying
		\begin{align*} 
				\tilde{\omega}(t)\in [0,1]\quad \forall t\in [0,T];\quad
				\tilde{\omega}(t)=1\quad \forall t\in I_c;\quad 
				\tilde{\omega}(t)=0\quad \forall t\in \bar{I}\backslash (\tilde{t}_1,\tilde{t}_2),
		\end{align*}
		and  $\tilde{\varphi}:=\bar \varphi\tilde{\omega}$, then $\tilde{\varphi}$ satisfies the following equation:
		\begin{align}
			\left\{
			\begin{aligned}
				-\partial_t \tilde{\varphi} - \Delta \tilde{\varphi} & =-\partial_t\tilde{\omega}\bar{\varphi}+\tilde{\omega}(\bar{u}-u_d)
				&&\mbox{in}\,\,\,\Omega\times(\tilde{t}_1,\tilde{t}_2),\\
				\tilde{\varphi}&=0 && \mbox{on}\,\, \Gamma\times (\tilde{t}_1,\tilde{t}_2),\\
				\tilde{\varphi}(\tilde{t}_2)&=0 && \mbox{in}\,\,\Omega.
			\end{aligned}
			\right.
		\end{align}
		Therefore, by using Lemma \ref{stateEexist1} we obtain $\tilde{\varphi}\in H^1((\tilde{t}_1,\tilde{t}_2);L^2(\Omega))\cap L^2((\tilde{t}_1,\tilde{t}_2);H^2(\Omega)\cap H^1_0(\Omega))$, which implies that $\bar{\varphi}\in H^1(I_c; L^2(\Omega))\cap L^2(I_c;H^2(\Omega)\cap H^1_0(\Omega))\hookrightarrow C(\bar{I}_c;H^1_0(\Omega))$ and the following estimate holds:
		\begin{eqnarray}
			&&\|\bar\varphi\|_{H^1(I_c; L^2(\Omega))}+\|\bar\varphi\|_{L^2(I_c;H^2(\Omega)\cap H^1_0(\Omega))}+\|\bar\varphi\|_{C(\bar{I}_c;H^1_0(\Omega))}\label{estimate_adjoint_c}\\
			&\le& \|\tilde\varphi\|_{H^1((\tilde{t}_1,\tilde{t}_2); L^2(\Omega))}+\|\tilde\varphi\|_{L^2((\tilde{t}_1,\tilde{t}_2);H^2(\Omega))}+\|\tilde\varphi\|_{C([\tilde{t}_1,\tilde{t}_2];H^1_0(\Omega))}\nonumber\\
			&\le&C(\|\bar\varphi\|_{L^2(I;L^2(\Omega))}+\|u_d\|_{L^2(I;L^2(\Omega))}+\|\bar u\|_{L^2(I;L^2(\Omega))})\nonumber\\
			&\le& C(\|f\|_{L^2(I;L^2(\Omega))}+\|u_d\|_{L^2(I;L^2(\Omega))}+\|u_T\|_{L^2(\Omega)}+\|u_0\|_{L^2(\Omega)}).\nonumber
		\end{eqnarray}
		
		Finally, we prove that $\bar{q}\in \mathcal M(\bar{I}_c;H^1(\omega))$. In view of Theorem \ref{Thm:Jordan}, there holds the relation $d\bar{q}=\bar{q}^\prime d|\bar{q}|$, where $\bar{q}^\prime(t,x)=-\frac{1}{\alpha}\bar{\varphi}(t,x)$ with $\bar\varphi\in C(\bar{I}_c;H_0^1(\Omega))$. Therefore, for any $\psi\in C(\bar{I}_c;(H^1(\omega))^*)$, there holds
		\begin{align*}
			\langle\bar q,\psi\rangle_{\bar I_c\times \omega}=&\int_{I_c}\langle\psi,d\bar{q}
			\rangle
			=-\frac{1}{\alpha}\int_{I_c}\langle\psi,\bar{\varphi}\rangle d|\bar{q}|\\
			\le&\frac{1}{\alpha}\|\psi\|_{C(\bar{I}_c;(H^1(\omega)^*))}\|\bar{\varphi}\|_{C(\bar{I}_c;H^1(\omega))}\|\bar{q}\|_{\mathcal M(\bar{I}_c;L^2(\omega))}.
		\end{align*}
		Thus, $\|\bar{q}\|_{\mathcal M(\bar{I}_c;H^1(\omega))}\le C\frac{1}{\alpha}\|\bar{\varphi}\|_{C(\bar{I}_c;H^1(\Omega))}\|\bar{q}\|_{\mathcal{M}(\bar{I}_c;L^2(\omega))}$ (cf. \cite{LM2009}). This combining with the estimate \eqref{estimate_adjoint_c} yields the conclusion. We thus finish the proof. \end{proof}
	\section{Finite element approximations}
	In this section we consider the space-time finite element approximation for optimal control problems. 
	
	\subsection{Notations for finite element methods}
	Let $\{\mathcal{T}_h\}_{h>0}$ be a family of quasi-uniform and shape regular triangulations of $\Omega$ in the sense of  Ciarlet~\cite{PGC1978}, such that $\bar\Omega=\cup_{K\in\mathcal{T}_h}\bar K$, where $h$ is the mesh parameter. Define the piecewise linear and continuous finite element space
	\begin{equation}
		V_h:=\left \{v_h\in H_0^1(\Omega):\ v_h|_K\in P_1(K),\ \forall K\in \mathcal{T}_h\right\},\nonumber
	\end{equation}
	where $ P_1(K)$ denotes the space of linear functions in $K$. 
	
	For simplicity, we assume that $\omega\subseteq\Omega$ is polygonal and the restriction of $\mathcal{T}_h$ on $\omega$ gives a partition of $\omega$. Thus, we define $U_h:=V_h|_\omega$ consisting of  piecewise linear and continuous functions in $\omega$.
	
	Next, we divide  $[0,T]$ into a family of subintervals $I_m:=(t_{m-1},t_m],\,m=1,2,\cdots,M-1,\,I_M:=(t_{M-1},t_M)$ with step size $\tau_m=t_m-t_{m-1}$ such that $[0,T]=\{0\}\cup\mathop{\cup}\limits_{i=1}^M I_m\cup\{T\}$, where  $0=t_0<t_1<\cdots <t_M=T$.  We assume that there exist $1<k_1<k_2<M$ such that $\bar{I}_c=[t_1,t_2]=\{t_{k_1-1}\}\cup \mathop{\cup}\limits_{m=k_1}^{k_2}I_{m}$, this can be achieved by setting $M$ sufficiently large. 
	The maximal time step is defined by  $\tau:=\mathop{\max}\limits_{1\le m\le M }{\tau_m}$. 
	
	Now, we are ready to define two time semi-discrete finite element spaces consisting of either piecewise constant or piecewise linear and continuous Ansatz. Define (cf. \cite{DanielsHinzeVierling2015})
	\begin{eqnarray}
		P_\tau:=\{v\in C(\bar I;H_0^1(\Omega)):\ v|_{I_m}\in \mathcal{P}_1(I_m;H_0^1(\Omega)),\ m=1,2,\dots,M \},\nonumber\\
		Y_\tau:=\{v\in L^2(I;H_0^1(\Omega)):\ v|_{I_m}\in \mathcal{P}_0(I_m;H_0^1(\Omega)),\ m=1,2,\dots,M,\ v(T)\in H_0^1(\Omega)  \},\nonumber
	\end{eqnarray}
	where $\mathcal{P}_i(I_m;H^1_0(\Omega))$ ($i=0,1$) denotes the set of polynomial functions of degree at most $i$ on time interval $I_m$ and valued in $H^1_0(\Omega)$, and let $$P_\tau^0:=\{v_\tau\in P_\tau:\ v_\tau(T)=0\}.$$
	
	The notation $\sigma=(\tau,h)$ denotes the vector of two discretization parameters $\tau$ and $h$. In order to introduce the Petrov-Galerkin scheme for parabolic equations, we  also need to define the following two time-space finite element spaces: 
	\begin{eqnarray}
		P_\sigma:=\left\{v\in P_\tau:\, v|_{I_m}\in \mathcal{P}_1(I_m;V_h),\ m=1,2,\dots,M \right\},\nonumber\\
		Y_\sigma:=\left\{v\in Y_\tau:\ v|_{I_m}\in \mathcal{P}_0(I_m;V_h),\ m=1,2,\dots,M,\ v(T)\in V_h \right\},\nonumber
	\end{eqnarray}
	where the definition of $\mathcal{P}_i(I_m;V_h)$ ($i=0,1$) is similar to $\mathcal{P}_i(I_m;H^1_0(\Omega))$ ($i=0,1$). We set $$P_\sigma^0:=\{v_\sigma\in P_\sigma:\ v_\sigma(T)=0\}.$$ It is clear that $P_\sigma\subset L^2(I;H^1_0(\Omega))\cap H^1(I;L^2(\Omega))$ and $Y_\sigma\subseteq L^2(I;H^1_0(\Omega))$.
	
	Let $\{x_j\}_{j=1}^{N_h}$ be the interior nodes of the mesh $\mathcal{T}_h$, and $\{\psi_{x_j}\}_{j=1}^{N_h}$, $\{e_{t_m}\}_{m=0}^M$ be the nodal basis functions of $\mathcal{P}_1$ elements of $H^1(\Omega)$ and $H^1(I)$, respectively.  Obviously, $P_\sigma$ and $Y_\sigma$ can be rewritten as 
	\begin{eqnarray}
		P_\sigma={\rm span}\{\psi_{x_j}\otimes e_{t_m}: 1\leq j\leq N_h,\ 0\leq m\leq M \},\nonumber\\
		Y_\sigma={\rm span}\{\psi_{x_j}\otimes \chi_{m}: 1\leq j\leq N_h,\ 1\leq m\leq M \},\nonumber
	\end{eqnarray}
	where $\chi_m$ denotes the characteristic function on $I_m$.  We define the indices
	\begin{eqnarray}
		\mathcal{I}_\sigma=\Big\{(j,m):(x_j,t_m)\in\bar\omega\times \bar I_c,1\leq j\leq N_h,\ 0\leq m\leq M \Big \}, \  \mathcal{I}_\tau = \big\{m:t_m\in \bar I_c, 0\leq m\leq M\big\},\nonumber
	\end{eqnarray}
	and the discrete spaces
	\begin{eqnarray}
		U_\tau ={\rm span}\{\delta_{t_m}:m\in \mathcal{I}_\tau\}\subset \mathcal{M}(\bar I_c),\quad V_\tau ={\rm span}\{e_{t_m}|_{\bar I_c}:m\in \mathcal{I}_\tau\}\subset {C}(\bar I_c),\nonumber\\
		U_\sigma ={\rm span}\{\psi_{x_j}|_{\bar\omega}\otimes \delta_{t_m}:(j,m)\in \mathcal{I}_\sigma\},\quad V_\sigma ={\rm span}\{\psi_{x_j}\otimes e_{t_m}|_{\bar\omega\times \bar I_c}:(j,m)\in \mathcal{I}_\sigma\}.\nonumber
	\end{eqnarray}
	For $q_\tau\in U_\tau$ and $v_\tau\in V_\tau$ we identify them with $\vec{q}_\tau=(q_1,\dots,q_{M_c})$ and $\vec{v}_\tau=(v_1,\dots,v_{M_c})$ where $M_c$ is the cardinality of $\mathcal{I}_\tau$. The linear functions in $V_\tau$ attain their maximum and minimum at the nodes. Therefore, for $v_\tau\in V_\tau$ we define
	\begin{eqnarray}
		\|v_\tau\|_{L^\infty(I_c)}=\max\limits_{1\leq m\leq M_c}|v_m|.\nonumber
	\end{eqnarray}
	Similarly, we have for $q_\tau\in U_\tau$ that
	\begin{eqnarray}
		\|q_\tau\|_{\mathcal{M}(\bar I_c)}=\sup\limits_{v\in {C}(\bar I_c),\|v\|_{L^\infty(I_c)}\le1}\sum\limits_{m=1}^{M_c}q_m\langle \delta_{t_m},v\rangle_{\bar I_c}=\sum\limits_{m=1}^{M_c}|q_m|.\nonumber
	\end{eqnarray}
	
	Given two functions $u\in L^2(I;H^1_0(\Omega))$ with $u(T)\in L^2(\Omega)$, $v\in L^2(I;H^1_0(\Omega))\cap H^1(I;H^{-1}(\Omega))$, we define a bilinear form $A(u,v)\to\mathbb{R}$ as follows:
	\begin{eqnarray}
		A(u,v):=-\langle u,\partial_t v\rangle_{L^2(I;H^1_0,H^{-1})}+\int_I(\nabla u(t),\nabla v(t))dt+(u(T),v(T)),\nonumber
	\end{eqnarray}
	where $\langle\cdot,\cdot\rangle_{L^2(I;H^1_0,H^{-1})}$ denotes the duality pairing between $L^2(I;H^1_0(\Omega))$ and $L^2(I;H^{-1}(\Omega))$.
	
	If, in addition, $u_\tau\in Y_\tau$ and $v_\tau\in P_\tau$, applying integration by parts to the bilinear form $A(u_\tau,v_\tau)$ we can obtain the following dual representation:
	\begin{equation}\label{dual_form}
		A(u_\tau,v_\tau)=\sum_{m=1}^M([u_\tau]_m,v_\tau(t_m))_{L^2(\Omega)}+(u_{\tau,0}^+,v_\tau(0))_{L^2(\Omega)}+\int_I(\nabla u_\tau(t),\nabla v_\tau(t))dt,
	\end{equation}
	where 
	\begin{align*}
		&u_{\tau,m+1}=u^+_{\tau,m}:=\mathop{\rm lim}\limits_{\epsilon\to 0^+}u_\tau(t_m+\epsilon),\quad
		u_{\tau,m}=u^-_ {\tau,m}:=\mathop{\rm lim}\limits_{\epsilon\to 0^+}u_\tau(t_m-\epsilon),\\
		&[u]_m:=u^+_{\tau,m}-u^-_{\tau,m},\ m=1,2,\cdots,M-1,\ [u_\tau]_M=u_\tau(T)-u^-_ {\tau,M},
	\end{align*}
	and $u_{\tau,m}:=u_\tau|_{I_m},\  m=1,2,\cdots,M$.
	
	In view of the identity \eqref{identity_2}, the very weak solution to the state equation \eqref{PDE_state} reads: Find $u\in L^2(I;H^1_0(\Omega))$ such that
	\begin{eqnarray}
		A(u,v)=\!\!\int_I(f,v)dt+\langle q,v\rangle_{\bar I_c\times \omega}+(u_0,v(x,0))\quad\forall v\in L^2(I;H^1_0(\Omega))\cap H^{1}(I;H^{-1}(\Omega)).\nonumber
	\end{eqnarray}
	Similarly, the weak formulation of the backward  equation (\ref{backward_PDE}) can be rewritten as: Find  $z\in L^2(I;H^1_0(\Omega))\cap H^{1}(I;H^{-1}(\Omega))$ such that
	\begin{eqnarray}\label{weak_formation_b}
		A(v,z)=\int_I(g,v)dt+(v(T),z_T)\quad \forall v\in L^2(I;H^1_0(\Omega))\cap H^{1}(I;H^{-1}(\Omega)).
	\end{eqnarray}
	
	\subsection{Interpolation and projection operators}
	In the following we introduce some interpolation operators defined in $\bar I_c$ and give their properties whose proofs are very similar to \cite[Theorem 3.1]{CasasClasonKunisch-2012}, see also \cite[Proposition 4.1]{CasasKunisch-2016}.
	\begin{lemma}
		Let the linear operators $\Lambda_\tau$ and $\Pi_\tau$ be defined as follows:	
		\begin{eqnarray}
			\Lambda_\tau:\mathcal{M}(\bar I_c)\rightarrow U_\tau\subset \mathcal{M}(\bar I_c),\quad \Lambda_\tau q:=\sum\limits_{m\in \mathcal{I}_\tau}\delta_{t_m}\int_{\bar I_c}e_{t_m}dq,\nonumber\\
			\Pi_\tau:{C}(\bar I_c)\rightarrow V_\tau\subset {C}(\bar I_c),\quad \Pi_\tau v:=\sum\limits_{m\in \mathcal{I}_\tau}v(t_m)e_{t_m}.\nonumber
		\end{eqnarray}
		Then for every $q\in \mathcal{M}(\bar I_c)$, $v\in {C}(\bar I_c)$ and $v_\tau\in V_\tau$, there hold
		\begin{eqnarray}
			\langle q,\Pi_\tau v \rangle_{\bar I_c}=\langle \Lambda_\tau q,v \rangle_{\bar I_c},\label{inter_2}\\
			\|\Lambda_\tau q\|_{\mathcal{M}(\bar I_c)}\leq \|q\|_{\mathcal{M}(\bar I_c)},\label{inter_3}\\
			\Lambda_\tau q\stackrel{*}{\rightharpoonup} q\ \mbox {\rm in} \ \mathcal{M}(\bar I_c)\quad \mbox{and}\ \|\Lambda_\tau q\|_{\mathcal{M}(\bar I_c)}\stackrel{\tau\rightarrow 0^+}{\longrightarrow} \| q\|_{\mathcal{M}(\bar I_c)},\label{inter_4}\\
			\|q-\Lambda_\tau q\|_{(H^{1}(I_c))'}\leq C\tau^{1\over 2} \|q\|_{\mathcal{M}(\bar I_c)},\quad \|q-\Lambda_\tau q\|_{(W^{1,\infty}(I_c))'}\leq C\tau  \|q\|_{\mathcal{M}(\bar I_c)}.\label{inter_10}
		\end{eqnarray}
	\end{lemma}
	\begin{proof}
		The proofs of \eqref{inter_2}-\eqref{inter_4} are trivial, thus we only provide the proof for  (\ref{inter_10}). For an arbitrary $v\in H^1( I_c)$, using (\ref{inter_2}) and the standard Lagrange interpolation error estimate we have for $s>1$ that
		\begin{eqnarray}
			\langle q-\Lambda_\tau q,v\rangle_{\bar I_c}=\langle q,v-\Pi_\tau v\rangle_{\bar I_c}\leq \|q\|_{\mathcal{M}(\bar I_c)}\|v-\Pi_\tau v\|_{{C}(\bar I_c)}\leq C\tau^{1-{1\over s}}\|q\|_{\mathcal{M}(\bar I_c)}\|v\|_{W^{1,s}(I_c)}.\nonumber
		\end{eqnarray}
		From the duality we can obtain two desired results by setting $s=2$ and $s=\infty$. 
	\end{proof}
	
	Let $\pi_h:\,L^2(\omega)\to V_h|_{\omega}$ be the usual $L^2$-projection defined by 
	$$(v-\pi_h v,\varphi_h)_{L^2(\omega)}=0\quad\forall\varphi_h\in V_h|_{\omega},$$
	then there holds
	$$
	\|v-\pi_hv\|_{L^2(\omega)}+h\|v-\pi_hv\|_{H^1(\omega)}\le Ch^m\|v\|_{H^m(\omega)}\quad \forall v\in H^m(\omega),\ m=1,2,
	$$
	where $C>0$ is a constant independent of $h$ and $v$. Note that the application of the operator $\pi_h$ to time-dependent arguments has to be understood pointwisely in time. Below, we will extend $\pi_h$ to a negative exponent Sobolev space that includes $L^2(\omega)$.

	Let $V:=H^1(\omega)$ (resp. $H_0^1(\Omega)$), $H:=L^2(\omega)$ (resp. $L^2(\Omega)$), then the inclusion $V\subseteq H$ is dense and continuous. Note that $V\hookrightarrow H=H^*\hookrightarrow V^*$ is a Gelfand triple, where $H\hookrightarrow V^*$ is given by $y\in H\to (y,\cdot)_{L^2(\Omega)}\in H^*\subseteq V^* $. Therefore, we extend in the following definition the usual projection $\pi_h$ from $H$ to $V^*$.
	
	\begin{definition}
		Define the action of the $L^2$-projection $\pi_h$ on $V^*$ as
		\begin{align*}
			\pi_h:V^*\to V_h|_{\omega}\quad
			v\to\pi_hv,
		\end{align*}
		where $\pi_hv\in V_h|_{\omega}$ satisfies
		\begin{equation}\label{identity_5}
			(\pi_hv,\varphi_h)_{H}=\langle v,\varphi_h\rangle_{V^*,V}\quad    \forall\varphi_h\in V_h|_{\omega}\subset V.
		\end{equation}	 
		Furthermore, we also define the following two interpolation operators: 	
		\begin{eqnarray}
			\Lambda_\sigma:\mathcal{M}(\bar I_c;H)\rightarrow U_\sigma\subset \mathcal{M}(\bar I_c;H),\quad \Lambda_\sigma q:=\pi_h(\Lambda_\tau q)=\sum\limits_{(j,m)\in \mathcal{I}_\sigma}q_{j,m}\delta_{t_m}\otimes\psi_{x_j}|_{\bar\omega},\nonumber
			\\
			\Pi_\sigma:{C}(\bar I_c;V^*)\rightarrow V_\sigma\subset {C}(\bar I_c;V^*),\quad \Pi_\sigma v:=\pi_h(\Pi_\tau v)=\sum\limits_{(j,m)\in \mathcal{I}_\sigma}v_{j,m}e_{t_m}\otimes\psi_{x_j}|_{\bar I_c\times\bar\omega},\nonumber
		\end{eqnarray}
		where $q_{j,m}:=\pi_h(\int_{\bar I_c}e_{t_m}dq)(x_j)$ and $v_{j,m}:=\pi_h(v(t_m))(x_j)$.
		
	\end{definition}
	
	It is easy to check that there exists a unique $\pi_hv\in V_h|_{\omega}$ such that the identity \eqref{identity_5} holds for any $v\in V^*$. Moreover, $\pi_h|_{H}:H\to V_h|_{\omega}$ is consistent with the usual $L^2$-projection and $\pi_h$ is stable, i.e.,
	\begin{align*}
		\|\pi_hv\|_{V^*}\le \|v\|_{V^*}, \quad
		\|\pi_h{v}\|_{H}\le \|{v}\|_{H}
	\end{align*}
	for $v\in V^*$ and  ${v}\in H$, respectively.
	The definition of above two interpolation operators  $\Lambda_\sigma$ and  $\Pi_\sigma$ are very similar to \cite[Theorem 4.2]{CasasClasonKunisch-2013}, see also \cite[Proposition 4.2]{CasasKunisch-2016}. 
	\begin{lemma}
		For every $q_\sigma\in U_\sigma$ and $v_\sigma\in V_\sigma$ one have
		\begin{eqnarray}
			\Lambda_\sigma q_\sigma=q_\sigma\quad\mbox{and}\quad \Pi_\sigma v_\sigma=v_\sigma.\label{inter_9}
		\end{eqnarray}
		Moreover, there hold
		\begin{eqnarray}
			\langle q,\Pi_\sigma v \rangle_{\bar I_c\times \omega}=\langle \Lambda_\sigma q,v \rangle_{\bar I_c\times \omega}
			\quad \forall\,(q,v)\in \mathcal{M}(\bar I_c;S) \times {C}(\bar I_c;S^*),
			\label{inter_5}
			\\
			\|\Lambda_\sigma q\|_{\mathcal{M}(\bar I_c;S)}\leq \|q\|_{\mathcal{M}(\bar I_c;S)},\quad \|\Pi_\sigma v\|_{{C}(\bar I_c;S^*)}\leq \|v\|_{{C}(\bar I_c;S^*)},\label{inter_7}\\
			\Lambda_\sigma q\stackrel{*}{\rightharpoonup} q\in \mathcal{M}(\bar I_c;H),\quad\ \|\Lambda_\sigma q\|_{\mathcal{M}(\bar I_c;H)}\stackrel{|\sigma|\to 0}{\longrightarrow} \| q\|_{\mathcal{M}(\bar I_c;H)},\label{inter_8}
		\end{eqnarray}
		where $S=H$ or $V$.
	\end{lemma}	
	\begin{proof}
		A simple calculation gives
		\begin{eqnarray}
			\langle q,\Pi_\sigma v \rangle_{\bar I_c\times \omega}&=&\langle q,\pi_h(\Pi_\tau v)\rangle_{\bar I_c\times \omega}  =\big\langle q,\pi_h\big(\sum_m v(t_m)e_{t_m}\big) \big\rangle_{\bar I_c\times \omega} \nonumber\\
			&=&\Big\langle\sum_m\pi_h (v(t_m)), \int_{I_m}e_{t_m}d q \Big\rangle_{S^*,S}\nonumber\\
			&=&\Big\langle\sum_m v(t_m), \pi_h(\int_{I_m}e_{t_m}d q) \Big\rangle_{S^*,S}\nonumber\\
			&=&\Big\langle v, \sum_m\pi_h(\int_{I_m}e_{t_m}d q)\otimes\delta_{t_m} \Big\rangle\nonumber\\
			&=&\langle \Lambda_\sigma q,v \rangle_{\bar I_c\times \omega}
		\end{eqnarray} 
		for any $(q,v) \in  \mathcal{M}(\bar I_c;S) \times  {C}(\bar I_c;S^*)$.
		Moreover, by using (\ref{inter_9}) and (\ref{inter_5}) we have $\langle q,v_\sigma \rangle_{\bar I_c\times \omega}=\langle q,\Pi_\sigma v_\sigma \rangle_{\bar I_c\times \omega}=\langle \Lambda_\sigma q,v_\sigma \rangle_{\bar I_c\times \omega}$.

		For the second inequality in (\ref{inter_7}), we use the stability of the projection $\pi_h$  to conclude  
		\begin{eqnarray}
			\|\Pi_\sigma v\|_{{C}(\bar I_c;S^*)}&=&\sup\limits_{t\in \bar I_c} \|\pi_h (\Pi_\tau v)(t)\|_{S^*}\leq\max\limits_{1\leq m\leq M_c} \|\pi_hv(t_m)\|_{S^*}\nonumber\\
			&\leq&\max\limits_{1\leq m\leq M_c} \|v(t_m)\|_{S^*}\leq \| v\|_{{C}(\bar I_c;{S^*})}.\nonumber
		\end{eqnarray}
		Next, we prove the first inequality in (\ref{inter_7})
		\begin{eqnarray}
			\|\Lambda_\sigma q\|_{\mathcal{M}(\bar I_c;S)}&=&\sup\limits_{\|v\|_{{C}(\bar I_c;S^*)}\leq 1}\langle\Lambda_\sigma q,v\rangle_{\bar I_c\times \omega} =\sup\limits_{\|v\|_{{C}(\bar I_c;S^*)}\leq 1}\langle q,\Pi_\sigma v\rangle_{\bar I_c\times \omega}\nonumber\\
			&\leq&\sup\limits_{\|v\|_{{C}(\bar I_c;S^*)}\leq 1}\|q\|_{\mathcal{M}(\bar I_c;S)}\|\Pi_\sigma v\|_{{C}(\bar I_c;S^*)}\nonumber\\
			&\leq&\|q\|_{\mathcal{M}(\bar I_c;S)}, \nonumber
		\end{eqnarray}
		where we have used the second inequality in (\ref{inter_7}). 
		This proves the first inequality in (\ref{inter_7}), while (\ref{inter_8}) is obvious. This finishes the proof.
	\end{proof}

	Next, we will introduce the following interpolation and projection operators in time (cf. \cite{DanielsHinzeVierling2015}). Define the $L^2$-projection  $\mathcal{P}_{Y_\tau}:L^2(I;H_0^1(\Omega))\rightarrow Y_\tau$ such that $\mathcal{P}_{Y_\tau} z\in Y_\tau$ satisfies
	\begin{align*}
		\mathcal{P}_{Y_\tau}z|_{I_m}:=\frac{1}{\tau_m}\int_{I_m}zdt\quad\forall z\in L^2(I;H_0^1(\Omega)),\ m=1,\cdots,M,\ \mbox{and}\ \mathcal{P}_{Y_\tau}v(T):=0.\nonumber
	\end{align*}
	In addition, we also need the following interpolation operators. Define the Lagrange interpolation operator ${\Pi}_{P_\tau}:C(\bar I;H_0^1(\Omega))\rightarrow P_\tau$ and the piecewise constant in time interpolation operator ${\Pi}_{Y_\tau}:C(\bar I;H_0^1(\Omega))\rightarrow Y_\tau$ such that ${\Pi}_{P_\tau} v\in P_\tau$,  ${\Pi}_{Y_\tau} v\in Y_\tau$  satisfy
	\begin{align*}
		{\Pi}_{P_\tau}v:=\sum\limits_{m=0}^Mv(t_m)e_{t_m},\quad
		{\Pi}_{Y_\tau}v:=\sum\limits_{m=1}^Mv(t_m)\chi_{I_m},\quad {\Pi}_{Y_\tau}v(T):=0,
	\end{align*}
	where $\{e_{t_m}\}$ and $\{\chi_{I_m}\}$ are the families of node basis functions of $\mathcal P_1$ and $\mathcal P_0$ elements on the time interval $I$, respectively. 
	\begin{lemma}\label{Lemma4}
		For arbitrary $ v\in L^2(I;H^1_0(\Omega))\cap H^1(I;L^2(\Omega))$, there hold the following standard interpolation error estimates (cf. \cite{DanielsHinzeVierling2015,Meidner-Vexler-2011}):
		\begin{eqnarray}
			\left\|v-\mathcal{P}_{Y_\tau}v\right\|_{L^2(I;L^2(\Omega))}\leq C\tau\left\|\partial_tv\right\|_{L^2(I;L^2(\Omega))},\\
			\|v-{\Pi}_{P_\tau}v\|_{L^2(I;L^2(\Omega))}\leq C\tau\|\partial_tv\|_{L^2(I;L^2(\Omega))},\label{estimate_12}\\
			\|v-{\Pi}_{P_\tau}v\|_{L^\infty(I;L^2(\Omega))}\leq C\tau^{1\over 2}\|\partial_tv\|_{L^2(I;L^2(\Omega))}.\label{estimate_13}
		\end{eqnarray}
		We also have the following half an order interpolation error estimates for any $v\in H^1(I;L^2(\Omega))\cap L^2(I;H^2(\Omega))$ (cf. \cite[Lemma 3.13]{KunischPieperVexler-2014}) 
		\begin{eqnarray}
			\|v-{\Pi}_{Y_\tau}v\|_{L^2(I;H_0^1(\Omega))}\leq C\tau^{1\over 2}(\|\partial_tv\|_{L^2(I;L^2(\Omega))}+\|\Delta
			v\|_{L^2(I;L^2(\Omega))}),\label{estimate_14}\\
			\|v-{\Pi}_{P_\tau}v\|_{L^2(I;H_0^1(\Omega))}\leq C\tau^{1\over 2}(\|\partial_tv\|_{L^2(I;L^2(\Omega))}+\|\Delta v\|_{L^2(I;L^2(\Omega))}),\label{estimate_15}
		\end{eqnarray}
		where $C>0$ is a constant independent of $\tau$ and $v$.
	\end{lemma}

	\subsection{Discretization of the optimal control problem}
	With the above preparation the discrete optimal control problem reads:
	\begin{align}\label{min-J-h}
		\min_{q\in \mathcal{M}(\bar I_c;L^2(\omega))\atop{u_\sigma\in Y_\sigma}}J_\sigma(u_\sigma,q)
		:=\frac{1}{2}\|u_\sigma-u_d\|_{L^2(I;L^2(\Omega))}^2
		+\frac{\beta}{2}\|u_\sigma(T)-u_T\|_{L^2(\Omega)}^2
		+\alpha\|q\|_{\mathcal{M}(\bar I_c;L^2(\omega))},
	\end{align}
	where $u_\sigma\in Y_\sigma$ is the discrete state variable satisfying the following discrete state equation:
	\begin{eqnarray}\label{discrete_state}
		A(u_\sigma,v_\sigma)=\int_I(f,v_\sigma)dt+\langle q,v_\sigma\rangle_{\bar I_c\times \omega}+(u_0,v_\sigma(0))\quad\forall v_\sigma\in P_\sigma.\label{state_h}
	\end{eqnarray}
	
	Note that in the above discrete optimal control problem \eqref{min-J-h} the control variable is not explicitly discretized, but in the following we will see that the discretization of the adjoint state indeed automatically yields the discretization of the control variable, which is the so-called variational discretization for optimal control problems proposed in \cite{Hinze-2005}.
	
	Similar to subsection \ref{subsec_1}, we can easily check that the discrete optimal control problem \eqref{min-J-h} exists at least one optimal pair. However, the optimal control is not unique in general, since the discrete control-to-state operator is not injective. In fact, we have $\langle\Lambda_\sigma q,v_{\sigma}\rangle_{\bar I_c\times \omega}=\langle q,v_{\sigma}\rangle_{\bar I_c\times \omega}$ for any $v_{\sigma}\in P_\sigma$, which implies that the cost functional $J_\sigma$ is not strictly convex on $\mathcal{M}(\bar I_c;L^2(\omega))$. Fortunately, we find  that the discrete optimal control problem \eqref{min-J-h} does indeed have a unique  optimal control  in the subspace $U_{\sigma}\subseteq\mathcal{M}(\bar I_c;L^2(\omega))$.

	To derive the discrete  first order optimality condition, we denote the solution of \eqref{state_h} by $u_{\sigma}(q)\in Y_{\sigma}$, then we can obtain the following reduced optimization problem:
	\begin{align}\label{discret_optimal_reduced}
		\min_{q\in \mathcal{M}(\bar I_c;L^2(\omega))}j_{\sigma}(q)
		:=j_{\sigma,1}(q)+j_{\sigma,2}(q),
	\end{align}
	where 
	\begin{align*}
		j_{\sigma, 1}(q):=\frac{1}{2}\|u_\sigma(q)-u_d\|_{L^2(I;L^2(\Omega))}^2
		+\frac{\beta}{2}\|u_\sigma(q)(T)-u_T\|_{L^2(\Omega)}^2,\quad
		j_{\sigma, 2}(q):=\alpha\|q\|_{\mathcal{M}(\bar I_c;L^2(\omega))}.
	\end{align*}
	Obviously, $j_{\sigma,1}:\mathcal M(\bar{I}_c;L^2(\omega))\to\mathbb{R}$ is a quadratic functional of tracking  type that is differentiable, and $j_{\sigma,2}:\mathcal M(\bar{I}_c;L^2(\omega))\to\mathbb{R}$ is convex and subdifferentiable. 
	
	By straightforward calculations, we obtain for any $q,\,p\in \mathcal{M}(\bar I_c;L^2(\omega))$ that
	\begin{align*}
		j^\prime_{\sigma,1}(q)p=\left(u_\sigma(q)-u_d,\tilde{u}_\sigma(p)\right)_{L^2(I;L^2(\Omega))}+\beta\left(u_\sigma(q)(T)-u_T,\tilde{u}_\sigma(p)(T)\right),
	\end{align*}
	where $\tilde{u}_\sigma(p)\in Y_\sigma$ is the finite element approximation of the state equation \eqref{PDE_state} with $f=0,\,u_0=0,\,q=p$, i.e., satisfying equation \eqref{state_h} with the assumed data. On the other hand, we define a discrete adjoint variable as follows: Find $\varphi_\sigma\in P_\sigma$ such that
	\begin{align}\label{adjoint_h}
		A(w_\sigma,\varphi_\sigma)=\int_I(u_\sigma(q)-u_d,w_\sigma)dt+\beta (u_\sigma(q)(\cdot,T)-u_T,w_\sigma(T))\quad\forall w_\sigma\in Y_\sigma.
	\end{align}
	Then, we obtain for any $q,\,p\in \mathcal{M}(\bar I_c;L^2(\omega))$ that
	\begin{align}\label{discretize_directive}
		j^\prime_{\sigma,1}(q)p&=\left(u_\sigma(q)-u_d,\tilde{u}_\sigma(p)\right)_{L^2(I;L^2(\Omega))}+\beta\left(u_\sigma(q)(T)-u_T,\tilde{u}_\sigma(p)(T)\right)\\\nonumber
		&=A(\tilde{u}_\sigma(p),\varphi_\sigma)=\langle p,\varphi_\sigma\rangle_{\bar I_c\times \omega},
	\end{align}
	which implies that $j^\prime_{\sigma,1}(q)=\varphi_\sigma|_{\bar I_c\times \omega}$ with $\varphi_\sigma$ given by \eqref{adjoint_h}.
	
	Now we are in the position to derive the first order  optimality condition for the discrete  optimization problem \eqref{discret_optimal_reduced}.
	\begin{theorem}
		Let $\hat{q}_\sigma\in \mathcal M(\bar{I}_c;L^2(\omega))$ be an optimal control and $\bar{u}_\sigma\in Y_{\sigma}$ be the corresponding optimal state of the optimal control problem (\ref{min-J-h}). Then there exists an adjoint state $\bar{\varphi}_\sigma\in P_\sigma$ solving \eqref{adjoint_h} with $u_{\sigma}(q)$ replaced by $\bar{u}_\sigma$ on the right-hand side, such that the following subgradient condition holds
		\begin{align}\label{disctere_condition_1}
			0\in \bar{\varphi}_{\sigma}|_{\bar{I}_c\times\omega}+\alpha\mathscr{\partial}\|\cdot\|_{\mathcal{M}(\bar{I}_c;L^2(\omega))}(\hat{q}_\sigma)
			\ \ \mbox{\rm in}\ \, (\mathcal{M}(\bar{I}_c;L^2(\omega)))^*,
		\end{align}
		i.e., 
		\begin{align}\label{disctere_condition_2}
			-\langle p-\hat{q}_\sigma,\bar{\varphi}_\sigma\rangle_{\bar I_c\times \omega}+\alpha\|\hat{q}\|_{\mathcal{M}(\bar{I}_c;L^2(\omega))}\le \alpha\|p\|_{\mathcal{M}(\bar{I}_c;L^2(\omega))}\quad\forall p\in \mathcal{M}(\bar{I}_c;L^2(\omega)),
		\end{align} 
		where $\mathscr{\partial}\|\cdot\|_{\mathcal{M}(\bar{I}_c;L^2(\omega))}(\hat{q}_\sigma)$ denotes the set of subgradients for $\|\cdot\|_{\mathcal{M}(\bar{I}_c;L^2(\omega))}$ at $\hat{q}_\sigma$ which is nonempty since  $\|\cdot\|_{\mathcal{M}(\bar{I}_c;L^2(\omega))}$ is convex on $\mathcal{M}(\bar{I}_c;L^2(\omega))$, and $(\mathcal{M}(\bar{I}_c;L^2(\omega)))^*$ denotes the topological dual of $\mathcal{M}(\bar{I}_c;L^2(\omega))$.
		
		Furthermore, from the above condition \eqref{disctere_condition_2} we can easily conclude the following relation between $\hat{q}_{\sigma}$ and $\bar{\varphi}_\sigma$:
		\begin{align}\label{OCP_OPT_h}
			\alpha \|\hat{q}_\sigma\|_{\mathcal{M}(\bar I_c;L^2(\omega))}+\langle \hat{q}_\sigma,\bar{\varphi}_\sigma\rangle_{\bar I_c\times \omega} =0,\\
			\|\bar{\varphi}_\sigma\|_{{C}(\bar I_c;L^2(\omega))}\left\{
			\begin{aligned}\label{adjoint_prop_h_0}
				&=\alpha && \mbox{if}\,\, \hat q_\sigma\neq 0,\\
				&\leq \alpha && \mbox{if}\,\, \hat q_\sigma=0.
			\end{aligned}
			\right.
		\end{align}
		
		In addition, since the discrete  control-to-state mapping has an infinite-dimensional kernel, the discrete optimal control problem (\ref{min-J-h}) admits more than one optimal control $\hat{q}_\sigma\in \mathcal{M}(\bar I_c;L^2(\omega))$ corresponding to the identical optimal state $\bar{u}_\sigma\in Y_{\sigma}$. Among  these optimal controls there exists a unique one $\bar{q}_\sigma\in U_\sigma$, such that for any other  optimal control $ \hat{q}_\sigma\in \mathcal{M}(\bar I_c;L^2(\omega))$ there holds $\bar{q}_\sigma =\Lambda_\sigma  \hat{q}_\sigma$, i.e., $$j_\sigma(\bar{q}_\sigma)=j_\sigma(\hat{q}_\sigma)=\min\limits_{q\in \mathcal{M}(\bar I_c;L^2(\omega))}j_\sigma(q).$$ In other words, $(\bar{q}_{\sigma},\bar{u}_\sigma)\in U_\sigma\times Y_\sigma$ is the unique optimal pair of the following fully discrete optimal control problem:
		\begin{align}\label{min-J_2}
			\min_{q_\sigma\in U_{\sigma}\atop{u_\sigma\in Y_\sigma}}J_\sigma(u_\sigma,q_\sigma)
			:=\frac{1}{2}\|u_\sigma-u_d\|_{L^2(I;L^2(\Omega))}^2
			+\frac{\beta}{2}\|u_\sigma(T)-u_T\|_{L^2(\Omega)}^2
			+\alpha\|q_\sigma\|_{\mathcal{M}(\bar I_c;L^2(\omega))},
		\end{align}
		where $u_{\sigma}\in Y_{\sigma}$ is the discrete state variable satisfying \eqref{discrete_state} with the discrete control $q_\sigma\in U_\sigma$, which can be computed in practice.		
	\end{theorem} 
	\begin{proof}
		Similar to Theorem \ref{Theorem_optimal_system}, we can easily obtain \eqref{disctere_condition_1}-\eqref{adjoint_prop_h_0} by recalling that 
		$j^\prime_{\sigma,1}(\bar{q}_\sigma)=\bar{\varphi}_\sigma|_{\bar{I}_c\times\omega}$ in \eqref{discretize_directive}. 
		
		Next, we state the non-uniqueness of optimal controls and the unique solvability of problem \eqref{min-J_2}. Let $\hat q_\sigma\in \mathcal{M}(\bar I_c;L^2(\omega))$ be any optimal control of problem (\ref{min-J-h}). Setting $\bar q_\sigma:=\Lambda_\sigma \hat q_\sigma$, it follows from (\ref{inter_5}) and (\ref{inter_7}) that $\langle \bar q_\sigma,v_\sigma\rangle_{\bar I_c\times \omega}=\langle \hat q_\sigma,v_\sigma\rangle_{\bar I_c\times \omega}$ for any $v_\sigma\in P_\sigma$ and
		\begin{eqnarray}
			J_\sigma( \bar q_\sigma)\leq J_\sigma( \hat q_\sigma).\nonumber
		\end{eqnarray}
		This means $\bar q_\sigma\in U_\sigma$ is also optimal, and $\bar{q}_\sigma\neq\hat q_\sigma$ unless $\hat q_\sigma\in U_\sigma$. On the other hand, the functional $j_\sigma$ is not strictly convex on $\mathcal{M}(\bar I_c;L^2(\omega))$. Therefore, the discrete optimization problem (\ref{min-J-h}) admits more than one solution in $\mathcal{M}(\bar I_c;L^2(\omega))$. 
		
		Obviously, the control-to-state mapping is injective on $U_\sigma$, and thus the discrete cost functional $J_\sigma$ is strictly convex on $U_\sigma$. Therefore, the optimization problem \eqref{min-J_2} admits a unique optimal pair $(\bar{q}_\sigma,\bar{u}_\sigma)\in U_{\sigma}\times Y_\sigma$. The uniqueness of optimal controls for the cost functional $J_\sigma$ in $U_\sigma$ ensures that any other optimal control $\hat q_\sigma\in \mathcal{M}(\bar I_c;L^2(\omega))$ satisfies $\bar q_\sigma =\Lambda_\sigma \hat q_\sigma$.  This finishes the proof.
	\end{proof}

	\section{Error estimates for the state and adjoint equations}
	In this section we intend to derive a priori error estimates for finite element solutions of the discrete state equation \eqref{discrete_state} and the discrete adjoint equation \eqref{adjoint_h}. Here we assume that the data of the optimal control problem \eqref{min-J}  satisfy the following assumptions.
	\begin{assumption}
		Assume that $f\in L^2(I;H^1(\Omega)),\ u_d\in L^2(I;L^2(\Omega)),\ u_0\in H^1_0(\Omega)$ and $ u_T\in H^1_0(\Omega)$.
	\end{assumption}
	
	In the following we also need the stability of the semi-discrete in time approximation to the state equation \eqref{PDE_state}. Therefore, we introduce here the semi-discrete state equation for given $f$, $u_0$ and $q$: Find $u_\tau\in Y_\tau$ such that
	\begin{eqnarray}
		A(u_\tau,v_\tau)=\int_I(f,v_\tau)dt+\langle q,v_\tau\rangle_{\bar I_c\times \omega}+(u_0,v_\tau(0))\quad\ \forall v_\tau\in P_\tau.\label{state_k}
	\end{eqnarray}
	Similarly, for any given $g$ and $z_T$, the semi-discrete approximation to problem (\ref{backward_PDE}) reads: Find $z_\tau\in P_\tau$ such that
	\begin{eqnarray}
		A(v_\tau,z_\tau)=\int_I(g,v_\tau)dt+(v_\tau(T),z_T)\quad\ \forall v_\tau\in Y_\tau.\label{backward_PDE_tau}
	\end{eqnarray}
	
	\subsection{Stability estimates for the discrete state and adjoint state}
	We first give a stability result for the semi-discrete approximation of the backward parabolic equation (\ref{backward_PDE}).
	{\begin{lemma}\label{Lm:adjoint_tau_stability}
			Let $z_\tau\in P_\tau$ solve (\ref{backward_PDE_tau}) for given  $g\in L^2(I;H^{-1}(\Omega))$ and $z_T\in L^2(\Omega)$, then there exists a constant $C>0$, independent of $\tau$, such that
			\begin{align}
				&\|z_\tau\|_{L^2(I;L^2(\Omega))}\leq C\sqrt{T}\|z_\tau\|_{C(\bar I;L^2(\Omega))},\nonumber\\
				&\|z_\tau\|_{C(\bar I;L^2(\Omega))}+\|\partial_tz_\tau\|_{L^2(I;H^{-1}(\Omega))}\leq C\|g\|_{L^2(I;H^{-1}(\Omega))}+\|z_T\|_{L^2(\Omega)}.\label{adjoint_k_stability_1}
			\end{align}
			In addition, if $g\in L^2(I;L^2(\Omega))$ and $z_T\in H^1_0(\Omega)$, then there hold
			\begin{align}
				&\|\nabla z_\tau\|_{L^2(I;L^2(\Omega))}\leq C\sqrt{T}\|\nabla z_\tau\|_{C(\bar I;L^2(\Omega))},\nonumber\\
				&\|\nabla z_\tau\|_{C (\bar I;L^2(\Omega))}+\|\partial_tz_\tau\|_{L^2(I;L^2(\Omega))}\leq C\|g\|_{L^2(I;L^2(\Omega))}+\|\nabla z_T\|_{L^2(\Omega)}.\label{adjoint_k_stability_2}
			\end{align}
	\end{lemma}}
	\begin{proof}
		For any fixed $m_0\in\{1,2,\dots,M\}$, setting $v_\tau|_{I_m}=0$ for $m=1,\cdots,m_0-1$, $v_\tau|_{I_m}=-\partial_t(-\Delta)^{-1}z_\tau|_{I_m}$ for $m=m_0,\cdots,M$ and $v_\tau(\cdot,T)=z_\tau(T)=z_T$ in (\ref{backward_PDE_tau}), it is clear that such $v_\tau\in Y_\tau$ and we have
		\begin{align}
			A(v_\tau,z_\tau)&=\|\nabla\partial_t(-\Delta)^{-1}z_\tau\|_{L^2(I';L^2(\Omega))}^2+{1\over 2}(\| z_\tau(t_{m_0-1})\|_{L^2(\Omega)}^2+\| z_\tau(T)\|_{L^2(\Omega)}^2)\nonumber\\
			&=-\int^T_{t_{m_0-1}}\langle g,\partial_t(-\Delta)^{-1}z_\tau\rangle dt+\|z_T\|^2_{L^2(\Omega)}\nonumber\\
			&\leq C\|g\|^2_{L^2(I';H^{-1}(\Omega))}+{1\over 2}\|\nabla\partial_t(-\Delta)^{-1}z_\tau\|_{L^2(I';L^2(\Omega))}^2+\|z_T\|^2_{L^2(\Omega)},\nonumber
		\end{align}
		where $I':=\cup_{m=m_0}^{M}I_m$, this implies that
		\begin{align}
			\|\nabla\partial_t(-\Delta)^{-1}z_\tau\|_{L^2(I';L^2(\Omega))}^2+\| z_\tau(t_{m_0-1})\|_{L^2(\Omega)}^2
			\leq C\|g\|^2_{L^2(I;H^{-1}(\Omega))}+\| z_T\|_{L^2(\Omega)}^2.\nonumber
		\end{align}
		That is, for any $t_m$ $(m=0,\cdots,M-1)$ we have the stability 
		\begin{eqnarray}
			\| z_\tau(t_m)\|_{L^2(\Omega)}\leq C\|g\|_{L^2(I;H^{-1}(\Omega))}+\| z_T\|_{L^2(\Omega)},\nonumber
		\end{eqnarray}
		i.e., $\|z_\tau\|_{L^\infty(I;L^2(\Omega))}\leq C\|g\|_{L^2(I;H^{-1}(\Omega))}+\| z_T\|_{L^2(\Omega)}$. This finishes the proof of \eqref{adjoint_k_stability_1}.
		
		Next, we prove the estimate \eqref{adjoint_k_stability_2}. Similar to the above procedure, for any fixed $m_0\in\{1,2,\cdots,M\}$, setting $v_\tau|_{I_m}=0$ for $m=1,\cdots,m_0-1$, $v_\tau|_{I_m}=-\partial_tz_{\tau}|_{I_m}$ for $m=m_0,\cdots,M$ and $v_\tau(\cdot,T)=z_\tau(T)=z_T$ in (\ref{backward_PDE_tau}), it is clear that  $v_\tau\in Y_\tau$ and we have
		\begin{align}
			A(v_\tau,z_\tau)&=\| \partial_tz_\tau\|_{L^2(I^\prime;L^2(\Omega))}^2+{1\over 2}(\|\nabla z_\tau(t_{m_0-1})\|_{L^2(\Omega)}^2-\|\nabla z_\tau(T)\|_{L^2(\Omega)}^2)+\| z_\tau(T)\|_{L^2(\Omega)}^2\nonumber\\
			&=-\int^T_{t_{m_0-1}}(g,\partial_tz_\tau)dt+\| z_\tau(T)\|_{L^2(\Omega)}^2\nonumber\\
			&\leq C\|g\|^2_{L^2(I^\prime;L^2(\Omega))}+{1\over 2}\|\partial_tz_\tau\|_{L^2(I^\prime;L^2(\Omega))}^2+\| z_\tau(T)\|_{L^2(\Omega)}^2.\nonumber
		\end{align}
		This implies that
		\begin{align}
			\|\partial_tz_\tau\|_{L^2(I^\prime;L^2(\Omega))}^2+\|\nabla z_\tau(t_{m_0-1})\|_{L^2(\Omega)}^2
			\leq C\|g\|^2_{L^2(I^\prime;L^2(\Omega))}+\| \nabla z_T\|_{L^2(\Omega)}^2,\nonumber
		\end{align}
		then the estimate \eqref{adjoint_k_stability_2} follows. This finishes the proof.
	\end{proof}
	
	Then we will prove the following stability result for the semi-discrete state approximation.
	\begin{lemma}\label{Lm:state_tau_stability}
			Let $u_\tau\in Y_\tau$ solve (\ref{state_k}) for given $f\in L^2(I;L^2(\Omega))$, $u_0\in L^2(\Omega)$ and $q\in \mathcal{M}(\bar I_c;L^2(\omega))$. Then there exists a constant $C>0$, independent of $\tau$, such that
			\begin{align}
				\|\nabla u_\tau\|_{L^2(I;L^2(\Omega))}+\|u_\tau(T)\|_{L^2(\Omega)}\leq C(\|f\|_{L^2(I;L^2(\Omega))}+\|q\|_{\mathcal{M}(\bar I_c;L^2(\omega))}+\|u_0\|_{L^2(\Omega)}).\label{state_k_stability}
			\end{align}
			Moreover, if $f\in L^2(I;H^1(\Omega))$, $u_0\in H^1_0(\Omega)$ and $q\in \mathcal{M}(\bar I_c;H^1(\omega))$, then there holds
			\begin{align}
				\|\Delta u_\tau\|_{L^2(I;L^2(\Omega))}+\|\nabla u_\tau(T)\|_{L^2(\Omega)}\leq C(\|f\|_{L^2(I;H^1(\Omega))}+\|q\|_{\mathcal{M}(\bar I_c;H^1(\omega))}+\|u_0\|_{H^1(\Omega)}),\label{state_k_stability_2}
			\end{align}
			where  $C>0$ is a constant independent of $\tau$.
	\end{lemma}
	\begin{proof}
		We first prove the estimate \eqref{state_k_stability}. Recall that $\|g\|^2_{L^2(I;H^{-1}(\Omega))}=\|\nabla u_\tau\|^2_{L^2(I;L^2(\Omega))}=
		\langle -\Delta u_\tau,u_\tau\rangle_{L^2(I;H^{-1},H^1)}$
		with $g=-\Delta u_\tau$, and $\|h\|_{L^2(\Omega)}=\| u_\tau(T)\|_{L^2(\Omega)}$ with $h=u_\tau(T)$. We denote by $z_\tau\in P_\tau$ the semi-discrete approximation defined in (\ref{backward_PDE_tau}) with $g=-\Delta u_\tau$ and $z_T= u_\tau(T)$, and test (\ref{backward_PDE_tau}) with $u_\tau$, then
		\begin{eqnarray}
			&&\int_I\langle g,u_\tau\rangle dt+(z_T,u_\tau(T))
			=A(u_\tau,z_\tau)\nonumber\\
			&=&\int_I(f,z_\tau)dt+\langle q,z_\tau\rangle_{\bar I_c\times \omega}+(u_0,z_\tau(0))\nonumber\\
			&\leq&\|f\|_{L^2(I;L^2(\Omega))}\|z_\tau\|_{L^2(I;L^2(\Omega))}+\|u_0\|_{L^2(\Omega)}\|z_\tau(0)\|_{L^2(\Omega)}+\|q\|_{\mathcal M(\bar I_c;L^2(\omega))}\|z_\tau\|_{{C}(\bar I_c;L^2(\omega))}\nonumber\\
			&\leq&C(\|f\|_{L^2(I;L^2(\Omega))}+\|u_0\|_{L^2(\Omega)}+\|q\|_{\mathcal{M}(\bar I_c;L^2(\omega))})\|z_\tau\|_{{C}(\bar I;L^2(\Omega))}\nonumber\\
			&\leq&C(\|f\|_{L^2(I;L^2(\Omega))}+\|u_0\|_{L^2(\Omega)}+\|q\|_{\mathcal{M}(\bar I_c;L^2(\omega))})(\|g\|_{L^2(I;H^{-1}(\Omega))}+\|z_T\|_{L^2(\Omega)}),\nonumber
		\end{eqnarray}
		where we used the scheme (\ref{state_k}) and
		the estimate (\ref{adjoint_k_stability_1})  in Lemma \ref{Lm:adjoint_tau_stability}. Therefore, we obtain the estimate \eqref{state_k_stability}.
		
		Now, we prepare to show \eqref{state_k_stability_2}. First, we denote by $u_{\tau,m}:=u_\tau|_{I_m}$, $m=1,2,\cdots,M$, and $u_{\tau,M+1}:=u_\tau(T)$. Thus there holds (cf. \cite{DanielsHinzeVierling2015})
		\begin{equation}\label{Crank_Nicolson}
			\begin{split}
				&(u_{\tau,1}-u_0,v)+{1\over2}\left(\tau_1\nabla u_{1},\nabla v\right)
				=( \tilde f_0,v)_{L^2(I_1;L^2(\Omega))},\\
				&(u_{\tau,m+1}-u_{\tau,m},v)+{1\over2}\left(\tau_{m+1}\nabla u_{\tau,m+1}+\tau_m\nabla u_{\tau,m},\nabla v\right)
				=(\tilde f_m,v),\\
				&(u_{\tau,M+1}-u_{\tau,M},v)+{1\over2}\left(\tau_{M}\nabla u_{M},\nabla v\right)
				=( \tilde f_M,v)
			\end{split}
		\end{equation}
		for arbitrary $ v\in H_0^1(\Omega)$, where 
		\begin{eqnarray}
			\tilde f_0=(f,e_{t_1})_{L^2(I_1)},\quad  \tilde f_M=(f,e_{t_M})_{L^2(I_M)},\nonumber\\
			\tilde f_m=(f,e_{t_m})_{L^2(I_m\cup I_{m+1})}
			+\int_{(I_m\cup I_{m+1})\cap \bar I_c}e_{t_m}dq(t)\in L^2(\Omega)\quad  m=1,2,\cdots,M-1.\nonumber
		\end{eqnarray}
		Since $f\in L^2(I;H^1(\Omega))$ and $q\in\mathcal{M}(\bar I_c;H^1(\omega))$, we have $\tilde f_m\in H^1(\Omega)$ for $m=0,1,\cdots,M$. Then from the first two expressions in (\ref{Crank_Nicolson})  it follows that $u_{\tau,m}\in H^2(\Omega)\cap H^1_0(\Omega)$ ($m=1,2,\cdots,M$) by the regularity of elliptic equations.
		By the last expression, there has $u_{\tau,M+1}=\tilde f_M+u_{\tau,M}+{1\over 2}\tau_M\Delta u_{\tau,M}$, which implies that $u_{\tau, M+1}\in L^{2}(\Omega)$. Therefore, we have $-\Delta u_{\tau,m}\in L^2(\Omega)$, $1\le m\le M$ and $-\Delta u_{\tau,M+1}\in H^{-2}(\Omega)$ in the sense of distributions. Using the similar idea we can show that  for any $g\in L^2(I;L^2(\Omega))$ and  $z_T\in H^1_0(\Omega)$, the semi-discrete solution of  (\ref{backward_PDE_tau}) satisfies $z_\tau\in L^2(I;H^2(\Omega))\cap C(\bar I;H^1(\Omega))$ and $z_\tau (0)\in H^1(\Omega)$.
		
		With the above spatial regularity of semi-discrete solutions to (\ref{backward_PDE_tau}) and (\ref{state_k}) we can rewrite the semi-discrete schemes (\ref{backward_PDE_tau}) and (\ref{state_k}) as:
		Find $u_\tau\in Y_\tau$ such that for any $ v_\tau\in P_\tau$
		\begin{equation}\label{A_Scheme_1}
			\int_I-(u_\tau,\partial_tv_\tau)-(\Delta u_\tau,v_\tau)dt+(u_\tau(T),v_\tau(T))=\int_I(f,v_\tau)dt+\langle q,v_\tau\rangle_{\bar I_c\times \omega}+(u_0,v_\tau(0)),	
		\end{equation}
		and find $z_\tau\in P_\tau$ such that for any $w_\tau\in Y_\tau$
		\begin{equation}\label{A_Scheme_2}
			\int_I-(w_\tau,\partial_tz_\tau)+( w_\tau,-\Delta z_\tau)dt+(w_\tau(T),z_\tau(T))=\int_I(g,w_\tau)dt+(z_T,w_\tau(T)).
		\end{equation}
		Since there are no spatial derivatives for the test functions in schemes \eqref{A_Scheme_1} and \eqref{A_Scheme_2}, the formulations \eqref{A_Scheme_1} and \eqref{A_Scheme_2} hold not only for all $v_\tau\in P_\tau$ and $w_\tau\in Y_\tau$,  but also hold, by the dense of $H^1_0(\Omega)$ in $L^2(\Omega)$, for all $v_\tau\in \tilde P_\tau$ and $w_\tau\in \tilde Y_\tau$, respectively, where 
		\begin{eqnarray}
			\tilde P_\tau:=\{v_\tau\in C(\bar I;L^2(\Omega)): v_\tau|_{I_m}\in \mathcal{P}_1(I_m;L^2(\Omega)),\ m=1,2,\cdots,M\},\nonumber\\
			\tilde Y_\tau:=\{v_\tau\in L^2( I;L^2(\Omega)): v_\tau|_{I_m}\in \mathcal{P}_0(I_m;L^2(\Omega)),m=1,2,\cdots,M,\ v_\tau(T)\in H^{-1}(\Omega)\}.\nonumber
		\end{eqnarray}
		
		We denote by $z_\tau$ the semi-discrete approximation to the backward equation (\ref{backward_PDE}) defined by (\ref{backward_PDE_tau}), or equivalently, \eqref{A_Scheme_2},  for arbitrary $g\in L^2(I;L^2(\Omega))$ and $z_T\in C_0^\infty(\Omega)$. Similarly, we test \eqref{A_Scheme_2} with $-\Delta u_\tau$, then 
		\vspace{-0.2cm}
		\begin{eqnarray}
			&&\int_I(g,-\Delta u_\tau)dt+\langle z_T,-\Delta u_\tau(T)\rangle
			=A(-\Delta u_\tau,z_\tau)\nonumber\\
			&=&A(u_\tau,-\Delta z_\tau)\nonumber\\
			&=&\int_I\langle f,-\Delta z_\tau\rangle dt+\langle q,-\Delta z_\tau\rangle_{\bar I_c\times \omega}+\langle u_0,-\Delta z_\tau(0)\rangle \nonumber\\
			&\leq&\|f\|_{L^2(I;H^1(\Omega))}\|\nabla z_\tau\|_{L^2(I;L^2(\Omega))}+\|\nabla u_0\|_{L^2(\Omega)}\|\nabla z_\tau(0)\|_{L^2(\Omega)}\nonumber\\
			&&+C\|q\|_{\mathcal{M}(\bar{I}_c;H^1(\omega))}\|\nabla z_\tau\|_{{C}(\bar I_c;L^2(\Omega))}\nonumber\\
			&\leq&C(\|f\|_{L^2(I;H^1(\Omega))}+\|u_0\|_{H^1(\Omega)}+\|q\|_{\mathcal{M}(\bar I_c;H^1(\omega))})\|\nabla z_\tau\|_{{C}(\bar I;L^2(\Omega))}\nonumber\\
			&\leq&C(\|f\|_{L^2(I;H^1(\Omega))}+\|u_0\|_{H^1(\Omega)}+\|q\|_{\mathcal{M}(\bar I_c;H^1(\omega))})(\|g\|_{L^2(I;L^2(\Omega))}+\|\nabla z_T\|_{L^2(\Omega)}),\nonumber
		\end{eqnarray}
		where we have used the scheme \eqref{A_Scheme_1} and the estimate \eqref{adjoint_k_stability_2} in Lemma \ref{Lm:adjoint_tau_stability}. 
By the density of $C_0^\infty(\Omega)$ in $H_0^1(\Omega)$ we obtain
		\begin{eqnarray}
		\|\Delta u_\tau\|_{L^2(I;L^2(\Omega))}+\|\Delta u_\tau(T)\|_{H^{-1}(\Omega)}\le C(\|f\|_{L^2(I;H^1(\Omega))}+\|u_0\|_{H^1(\Omega)}+\|q\|_{\mathcal{M}(\bar I_c;H^1(\omega))}),\nonumber
		\end{eqnarray}
		which confirms the estimate \eqref{state_k_stability_2} by using the fact $\|\Delta u_\tau(T)\|_{H^{-1}(\Omega)}\approx \|\nabla u_\tau(T)\|_{L^2(\Omega)}$. This finishes the proof.
	\end{proof}	
	
	Furthermore, the fully discrete finite element approximation of the backward parabolic equation (\ref{backward_PDE}) can be defined as: Find $z_\sigma\in P_\sigma$ such that
	\begin{eqnarray}
		A(v_\sigma,z_\sigma)=\int_I(g,v_\sigma)dt+(v_\sigma(T),z_T)\quad\ \forall v_\sigma\in Y_\sigma.\label{backward_PDE_sigma}
	\end{eqnarray}
	
	Similar to Lemma \ref{Lm:adjoint_tau_stability} we have the following stability result on the fully discrete approximation of backward parabolic equations.
	\begin{lemma}\label{Lm:adjoint_stability}
		Let $z_\sigma\in P_\sigma$ solve (\ref{backward_PDE_sigma}) for given $g\in L^2(I;H^{-1}(\Omega))$ and $z_T\in {L^2(\Omega)}$. Then there exists a constant $C>0$, independent of $\sigma$, such that
		\begin{align}
			\left\|\nabla\partial_t(-\Delta_h)^{-1}z_\sigma\right\|_{L^2(I;L^2(\Omega))}+\| z_\sigma\|_{C(\bar{I};L^2(\Omega))}\leq C\left(\|g\|_{L^2(I;H^{-1}(\Omega))}+\| z_T\|_{L^2(\Omega)}\right),\label{adjoint_h_stability}
		\end{align}	
	where $-\Delta_h:V_h\rightarrow V_h$ is defined as $(-\Delta_hv_h,w_h):=(\nabla v_h,\nabla w_h)$ for any $w_h\in V_h$.
	\end{lemma}	
	\begin{proof}
		Setting  $v_\sigma\in Y_\sigma$ such that  $v_\sigma|_{I_m}=-\partial_t(-\Delta_h)^{-1}z_\sigma|_{I_m}\in Y_\sigma$ ($m=m_0,\cdots,M$), $v_\sigma|_{I_m}=0\ (m=1,2,\cdots,m_0-1)$  and $v_\sigma(T)=z_\sigma(T)$ in (\ref{backward_PDE_sigma}) for arbitrary $1\le m_0\le M$, there holds
		\begin{align}
			A(v_\sigma,z_\sigma)&=\|\nabla\partial_t(-\Delta_h)^{-1}z_\sigma\|_{L^2(I^\prime;L^2(\Omega))}^2+{1\over 2}(\| z_\sigma(t_{m_0-1})\|_{L^2(\Omega)}^2+\| z_\sigma(T)\|_{L^2(\Omega)}^2)\nonumber\\
			&=\int_{t_{m_0-1}}^T\langle g,-\partial_t(-\Delta_h)^{-1}z_\sigma\rangle dt+(z_T,z_\sigma(T))\nonumber\\
			&\leq C\|g\|^2_{L^2(I^\prime;H^{-1}(\Omega))}+{1\over 2}\|\nabla\partial_t(-\Delta_h)^{-1}z_\sigma\|_{L^2(I^\prime;L^2(\Omega))}^2+C\|z_T\|^2_{L^2(\Omega)}+{1\over 2}\|z_\sigma(T)\|^2_{L^2(\Omega)},\nonumber
		\end{align}
		which implies that
		\begin{align}
			\|\nabla\partial_t(-\Delta_h)^{-1}z_\sigma\|_{L^2(I^\prime;L^2(\Omega))}+\| z_\sigma(t_{m_0-1})\|_{L^2(\Omega)}
			\leq C(\|g\|_{L^2(I^\prime;H^{-1}(\Omega))}+\|z_T\|_{L^2(\Omega)})\nonumber
		\end{align}
		for any $1\le m_0\le M$. Therefore, we can derive the conclusion. This finishes the proof. 
	\end{proof}

	Now we prove the following stability for the fully discrete approximation to the state equation.
	\begin{lemma}\label{Lm:state_h_stability}
		For given $f\in L^2(I;L^2(\Omega))$, $u_0\in L^2(\Omega)$ and $q\in \mathcal{M}(\bar I_c;L^2(\omega))$, let $u_\sigma(q)\in Y_\sigma$ solve (\ref{state_h}). Then there holds the following estimate:
		\begin{align}
			\|\nabla u_\sigma(q)\|_{L^2(I;L^2(\Omega))}+\|u_\sigma(q)(T)\|_{L^2(\Omega)}\leq C(\|f\|_{L^2(I;L^2(\Omega))}+\|q\|_{\mathcal{M}(\bar I_c;L^2(\omega))}+\|u_0\|_{L^2(\Omega)}),\nonumber\\\label{state_h_stability}
		\end{align}
		\vspace{0.2cm}
		where $C>0$ is a constant independent of $\sigma$.
	\end{lemma}
	
	\begin{proof}
		Similar to Lemma \ref{Lm:state_tau_stability}, we denote by $z_\sigma$ the discrete approximation defined in  (\ref{backward_PDE_sigma}) with $g=u_\sigma$ and $z_T=u_\sigma(T)$, then  
		\vspace{-0.1cm}
		\begin{eqnarray}
			&&\int_I(g,u_\sigma)dt+(z_T,u_\sigma(T))
			=A(u_\sigma,z_\sigma)\nonumber\\
			&=&\int_I(f,z_\sigma)dt+\langle q,z_\sigma\rangle_{\bar I_c\times \omega}+(u_0,z_\sigma(0))\nonumber\\
			&\leq&\|f\|_{L^2(I;L^2(\Omega))}\|z_\sigma\|_{L^2(I;L^2(\Omega))}+\|u_0\|_{L^2(\Omega)}\|z_\sigma(0)\|_{L^2(\Omega)}+\|q\|_{\mathcal{M}(\bar I_c;L^2(\omega))}\|z_\sigma\|_{{C}(\bar I_c;L^2(\omega))}\nonumber\\
			&\leq&C(\|f\|_{L^2(I;L^2(\Omega))}+\|u_0\|_{L^2(\Omega)}+\|q\|_{\mathcal{M}(\bar I_c;L^2(\omega))})\|z_\sigma\|_{{C}(\bar I;L^2(\Omega))}\nonumber\\
			&\leq&C(\|f\|_{L^2(I;L^2(\Omega))}+\|u_0\|_{L^2(\Omega)}+\|q\|_{\mathcal{M}(\bar I_c;L^2(\omega))})(\|g\|_{L^2(I;H^{-1}(\Omega))}+\|z_T\|_{L^2(\Omega)}),\nonumber
		\end{eqnarray}
		where  we have used  Lemma \ref{Lm:adjoint_stability}. Then we can obtain the result by canceling the common term. This finishes the proof.
	\end{proof}
	\subsection{A priori error estimates for the state and adjoint equations}
	In this subsection we are now able to give a priori error estimates for the finite element solutions to the state and adjoint equations.
	\begin{theorem}\label{Thm:state_error_1}
		For arbitrary $f\in L^2(I;H^1(\Omega))$, $q\in\mathcal{M}(\bar I_c;H^1(\omega))$ and $u_0\in H^1_0(\Omega)$, let $u\in L^2(I;L^2(\Omega))$ be the solution to problem (\ref{PDE_state}) and $u_\sigma\in Y_\sigma$ be its discretization defined in (\ref{discrete_state}). Then there exists a positive constant $C$, independent of $\sigma$, such that
		\begin{eqnarray}
			&&\|u-u_\sigma(q)\|_{L^2(I;L^2(\Omega))}+\|(u-u_\sigma(q))(T)\|_{L^2(\Omega)}\nonumber\\
			&\leq& C(h+\tau^{1\over 2})(\|f\|_{L^2(I;H^1(\Omega))}+\|q\|_{\mathcal{M}(\bar I_c;H^1(\omega))}+\|u_0\|_{H^1(\Omega)}).\label{estimate_u}
		\end{eqnarray}
	\end{theorem}
	\begin{proof}
		We split the fully discrete error estimate into the temporal and spatial parts. To begin with, let  $u_\tau\in Y_\tau$ be the semi-discrete solution to problem (\ref{state_k}). Then we estimate respectively $\|u-u_\tau\|_{L^2(I;L^2(\Omega))}+\|u(T)-u_\tau(T)\|_{L^2(\Omega)}$ and $\|u_\tau-u_\sigma\|_{L^2(I;L^2(\Omega))}+ \|u_\tau(T)-u_\sigma(T)\|_{L^2(\Omega)}$.  To prove the two estimates  we use the duality argument (cf.  \cite{CasasClasonKunisch-2013,Gong-2013}).
		
		We first prove the estimate for $\|u-u_\tau\|_{L^2(I;L^2(\Omega))}+\|u(T)-u_\tau(T)\|_{L^2(\Omega)}$. Let $z\in H^1(I;L^2(\Omega))\cap L^2(I;H^2(\Omega)\cap H^1_0(\Omega))$ be the solution to the backward parabolic equation (\ref{backward_PDE}) with $g:=u-u_\tau,\ z_T:=u(T)-u_\tau(T)$. Setting $\tilde z_\tau :={\Pi}_{P_\tau} z\in P_\tau$,  then there holds $(z-\tilde z_\tau)(T)=0$ and $$
		(v_\tau,\partial_t(z-\tilde z_\tau))_{L^2(I;L^2(\Omega))}=0\quad\forall v_\tau\in Y_\tau.
		$$
		From (\ref{weak_formation_b}) and (\ref{state_k}) we have
		\begin{eqnarray}
			&&\int_I(g,u-u_\tau)dt+(z_T,u(T)-u_\tau(T))
			=A(u,z)-A(u_\tau,z_\tau)\nonumber\\
			&=&A(u,z)-A(u_\tau,z)-\int_I(f,\tilde z_\tau)dt-\langle q,\tilde z_\tau\rangle_{\bar I_c\times \omega}-(u_0,\tilde z_\tau(0))+A(u_\tau,\tilde z_\tau)\nonumber\\
			&=&\int_I(f,z-\tilde z_\tau)dt+\langle q,z-\tilde z_\tau\rangle_{\bar I_c\times \omega}+(u_0,z(0)-\tilde z_\tau(0))-A(u_\tau,z-\tilde z_\tau)\nonumber\\
			&=&\int_I(f,z-\tilde z_\tau)dt+\langle q,z-\tilde z_\tau\rangle_{\bar I_c\times \omega}-\int_I(\nabla u_\tau,\nabla(z-\tilde{z}_\tau))dt\nonumber\\
			&\le&\left(\|f\|_{L^2(I;L^2(\Omega))}+\|\Delta u_\tau\|_{L^2(I;L^2(\Omega))}\right)\|z-\tilde{z}_\tau\|_{L^2(I;L^2(\Omega))}+\|q\|_{\mathcal{M}(\bar{I}_c;L^2(\omega))}\|z-\tilde{z}_\tau\|_{C(\bar{I}_c;L^2(\omega))}\nonumber\\
			&\le&C\tau\left(\|f\|_{L^2(I;L^2(\Omega))}+\|\Delta u_\tau\|_{L^2(I;L^2(\Omega))}\right)\|z\|_{H^1(I;L^2(\Omega))}+C\tau^{\frac{1}{2}}\|q\|_{\mathcal{M}(\bar{I}_c;L^2(\omega))}\|z\|_{H^1({I}_c;L^2(\Omega))}\nonumber\\
			&\le&C\tau\left(\|f\|_{L^2(I;L^2(\Omega))}+\|\Delta u_\tau\|_{L^2(I;L^2(\Omega))}\right)\left(\|g\|_{L^2(I;L^2(\Omega))}+\|z_T\|_{H^1(\Omega)}\right)\nonumber\\
			&&+C\tau^{\frac{1}{2}}\|q\|_{\mathcal{M}(\bar{I}_c;L^2(\omega))}\left(\|g\|_{L^2(I;L^2(\Omega))}+\|z_T\|_{L^2(\Omega)}\right) \nonumber\\
			&\le&C\tau H\left(\|g\|_{L^2(I;L^2(\Omega))}+H\right)+C\tau^{\frac{1}{2}}\|q\|_{\mathcal{M}(\bar{I}_c;L^2(\omega))}\left(\|g\|_{L^2(I;L^2(\Omega))}+\|z_T\|_{L^2(\Omega)}\right) \nonumber\\
			&\le&C\tau(1+\tau) H^2+C\tau\|q\|^2_{\mathcal{M}(\bar{I}_c;L^2(\omega))}+{1\over 2}(\|g\|^2_{L^2(I;L^2(\Omega))}+\|z_T\|^2_{L^2(\Omega)}),\nonumber  
		\end{eqnarray}
		where we have used Lemma \ref{Lm:state_tau_stability}, Theorem \ref{Thm:existence_state} and
		$H:=\|f\|_{L^2(I;H^1(\Omega))}+\|q \|_{\mathcal{M}(\bar{I}_c;H^1(\omega))}+\|u_0\|_{H^1(\Omega)}$. Therefore, we have obtained 
		\begin{eqnarray}
			&&\|u-u_\tau\|_{L^2(I;L^2(\Omega))}+\|u(T)-u_\tau(T)\|_{L^2(I;L^2(\Omega))}\\\label{key_0}
			&\le& C\tau^{\frac{1}{2}}(\|f\|_{L^2(I;H^1(\Omega))}+\|q \|_{\mathcal{M}(\bar{I}_c;H^1(\omega))}+\|u_0\|_{H^1(\Omega)}).\nonumber
		\end{eqnarray}
		
		Next, we estimate $\|u_\tau-u_\sigma\|_{L^2(I;L^2(\Omega))}+\|u_\tau(T)-u_\sigma(T)\|_{L^2(\Omega)}$.
		Note that there exists the following splitting:
		\begin{eqnarray}
			u_\tau-u_\sigma=u_\tau-\mathcal{R}_hu_\tau+\mathcal{R}_hu_\tau-u_\sigma:=\eta_\sigma+\xi_\sigma,\nonumber
		\end{eqnarray}
		where $\mathcal{R}_h:H_0^1(\Omega)\rightarrow V_h$ is the standard spatial Ritz projection (cf. \cite{PGC1978}).
		
		Let $z_\sigma\in P_\sigma$ be the fully discrete solution to problem (\ref{backward_PDE_sigma}) with $g:=\xi_\sigma$ and $z_T:=\xi_\sigma(T)$.  Taking $v_\sigma=\xi_\sigma \in Y_\sigma$ in the scheme (\ref{backward_PDE_sigma}) and applying the Galerkin orthogonality, one obtains
		\begin{eqnarray}\label{key_00}
			&&\int_I(g,\xi_\sigma)dt+(z_T,\xi_\sigma(T)) 
			=A(\xi_\sigma,z_\sigma)=-A(\eta_\sigma,z_\sigma)\nonumber\\ 
			&=&\int_I(\eta_\sigma,\partial_t z_\sigma)-(\nabla\eta_\sigma,\nabla z_\sigma)dt-(\eta_\sigma(T),z_\sigma(T))\nonumber\\
			&=&\int_I(\eta_\sigma,\partial_t z_\sigma)dt-(\eta_\sigma(T),z_\sigma(T))\nonumber\\
			&\le&\|\nabla\eta_\sigma\|_{L^2(I;L^2(\Omega))}\|\nabla\partial_t(-\Delta_h)^{-1}z_\sigma\|_{L^2(I;L^2(\Omega))}+\|\eta_\sigma(T)\|_{L^2(\Omega)}\| z_\sigma(T)\|_{L^2(\Omega)}\nonumber\\
			&\le&C(\|\nabla\eta_\sigma\|_{L^2(I;L^2(\Omega))}+\|\eta_\sigma(T)\|_{L^2(\Omega)})(\|g\|_{L^2(I;L^2(\Omega))}+\|z_T\|_{L^2(\Omega)}),\nonumber
		\end{eqnarray}
		where we have used Lemma \ref{Lm:adjoint_stability}. Therefore, we have
		\begin{eqnarray}
			&&\|u_\tau-u_\sigma\|_{L^2(I;L^2(\Omega))}+\|u_\tau(T)-u_\sigma(T)\|_{L^2(\Omega)}\\\label{key_1}
			&\le& C(\|\nabla\eta_\sigma\|_{L^2(I;L^2(\Omega))}+\|\eta_\sigma(T)\|_{L^2(\Omega)})\nonumber\\
			&\le& Ch(\|\Delta u_\tau\|_{L^2(I;L^2(\Omega))}+\|\nabla u_\tau(T)\|_{L^2(\Omega)})\nonumber\\
			&\le& Ch(\|f\|_{L^2(I;H^1(\Omega))}+\|q\|_{\mathcal{M}(\bar{I}_c;H^1(\omega))}+\|u_0\|_{H^1(\Omega)}),\nonumber
		\end{eqnarray}
		where we have used Lemma \ref{Lm:state_tau_stability}. Combining the above two estimates we  finish the proof.
	\end{proof}

	\begin{theorem}\label{estimate_adjoint}
		For any $z_T\in H^1_0(\Omega)$ and $g\in L^2(I;L^2(\Omega))$, let $z_\sigma\in P_\sigma$ be the solution of the discrete scheme  \eqref{backward_PDE_sigma}, and $z\in H^1(I;L^2(\Omega))\cap L^2(I;H^2(\Omega)\cap H^1_0(\Omega))$ be the solution of equation \eqref{backward_PDE}. Then there exists a positive constant $C>0$, independent of $\sigma$, such that
		\begin{equation}\label{error_estimate_adjoint}
			\|z-z_\sigma\|_{C(\bar{I};L^2(\Omega))}\le C(h+\tau^{\frac{1}{2}})(\|g\|_{L^2(I;L^2(\Omega))}+\|z_T\|_{H^1(\Omega)}).
		\end{equation}
	\end{theorem}
	\begin{proof}
		Let $e_{\sigma}:=z-z_\sigma=(z-\pi_{h}\Pi_{P_\tau }z)+(\pi_{h}\Pi_{P_\tau }z-z_\sigma)=:\eta_\sigma+\zeta_\sigma$, then by the Galerkin orthogonality there holds for any $v_\sigma\in Y_\sigma$ that
		\begin{eqnarray}
			A(v_\sigma,\zeta_\sigma)=-A(v_\sigma,\eta_\sigma)&=&\int_I(v_\sigma,\partial_t\eta_\sigma)-(\nabla v_\sigma,\nabla\eta_\sigma)dt-(v_\sigma(T),\eta_\sigma(T))\nonumber\\
			&=&\sum_{m=1}^M\Big(v_\sigma|_{I_m},(z-\pi_{h}\Pi_{P_\tau }z)(t_m)-(z-\pi_{h}\Pi_{P_\tau }z)(t_{m-1})\Big)\\
			&&-(v_\sigma(T),\eta_\sigma(T))-\int_I(\nabla v_\sigma,\nabla\eta_\sigma)dt\nonumber\\
			&=&-\int_I(\nabla v_\sigma,\nabla\eta_\sigma)dt,\nonumber
		\end{eqnarray}
		i.e., $\zeta_\sigma$ satisfies the following variational problem: Find $\zeta_\sigma\in P^0_\sigma$ such that
		\begin{eqnarray}\label{pointwise}
			\int_I-(v_\sigma,\partial_t\zeta_\sigma)+(\nabla v_\sigma,\nabla \zeta_\sigma)dt=-\int_I(\nabla v_\sigma,\nabla\eta_\sigma)dt\quad\ \forall v_\sigma\in Y_\sigma.
		\end{eqnarray}
		For arbitrary $1\le m_0\le M$, taking the test function $v_\sigma$ satisfying  $v_\sigma|_{I_m}=-\partial_t(-\Delta_h)^{-1}\zeta_\sigma|_{I_m}$, $m=m_0,m_0+1,\cdots,M$ and $v_\sigma|_{I_m}=0$, $m=1,2,\cdots,m_0-1$, $v_\sigma(T)=0$ in the above identity \eqref{pointwise}, then we have
		\begin{eqnarray}
			&&\left\|\nabla \partial_t(-\Delta_h)^{-1}\zeta_\sigma\right\|^2_{L^2(I^\prime;L^2(\Omega))}+\frac{1}{2}\|\zeta_\sigma(t_{m_0-1})\|^2_{L^2(\Omega)}\nonumber\\
			&=&\int_{t_{m_0-1}}^T(\nabla\partial_t(-\Delta_h)^{-1}\zeta_\sigma,\nabla\eta_\sigma)dt\nonumber\\
			&\le&\frac{1}{2}\|\nabla\partial_t(-\Delta_h)^{-1}\zeta_\sigma\|^2_{L^2({I^\prime};L^2(\Omega))}+\frac{1}{2}\|\nabla\eta_\sigma\|^2_{L^2({I^\prime};L^2(\Omega))}.\nonumber
		\end{eqnarray}
		Therefore, we obtain
		\begin{eqnarray}
			\|\zeta_\sigma(t_{m_0-1})\|^2_{L^2(\Omega)}\le \|\nabla\eta_\sigma\|^2_{L^2({I^\prime};L^2(\Omega))}\nonumber
		\end{eqnarray}
		for arbitrary $1\le m_0\le M$, i.e.,
		\begin{eqnarray}
			\|\zeta_\sigma\|_{C(\bar{I};L^2(\Omega))}=\max_{1\le m\le M-1} \|\zeta_\sigma(t_m)\|_{L^2(\Omega)}\le \|\nabla\eta_\sigma\|_{L^2(I;L^2(\Omega))}.\nonumber
		\end{eqnarray}
		Combining with the expression $e_\sigma=\eta_\sigma+\zeta_\sigma$ one deduces
		\begin{eqnarray}
			\|e_\sigma\|_{C(\bar{I};L^2(\Omega))}\le\|\eta_\sigma\|_{C(\bar{I};L^2(\Omega))}+\|\nabla\eta_\sigma\|_{L^2(I;L^2(\Omega))}.\label{pointwise_estimate}
		\end{eqnarray}
		Then it suffices to bound the two terms on the right-hand side. 
		
		The first term on the right-hand side of \eqref{pointwise_estimate} can be bounded by 
		\begin{eqnarray}
			\|\eta_\sigma\|_{C(\bar{I};L^2(\Omega))}&=&\|z-\pi_h\Pi_{P_\tau}z\|_{C(\bar{I};L^2(\Omega))}\nonumber\\ 
			&\le& \|z-\pi_hz\|_{C(\bar{I};L^2(\Omega))}+\|\pi_h(z-\Pi_{P_\tau}z)\|_{C(\bar{I};L^2(\Omega))}\nonumber\\
			&\le& \|z-\pi_hz\|_{C(\bar{I};L^2(\Omega))}+C\|z-\Pi_{P_\tau}z\|_{C(\bar{I};L^2(\Omega))}\nonumber\\
			&\le& Ch\|z\|_{C(\bar{I};H^1(\Omega))}+C\tau^{\frac{1}{2}}\|z\|_{H^1(I;L^2(\Omega))}\nonumber\\
			&\le& C(h+\tau^{\frac{1}{2}})(\|z\|_{L^2(I;H^2(\Omega))}+\|z\|_{H^1(I;L^2(\Omega))})\nonumber\\
			&\le&
			C(h+\tau^{\frac{1}{2}})(\|g\|_{L^2(I;L^2(\Omega))}+\|z_T\|_{H^1(\Omega)}),\label{key_01}
		\end{eqnarray}
		where we have used the stability of the $L^2$-projection $\pi_h$. On the other hand, $\|\nabla\eta_\sigma\|_{L^2(I;L^2(\Omega))}$ can be  estimated by
		\begin{eqnarray}
			\|\nabla\eta_\sigma\|_{L^2(I;L^2(\Omega))}&=&\|\nabla(z-\pi_h\Pi_{P_\tau}z)\|_{L^2(I;L^2(\Omega))}\nonumber \\ 
			&\le& \|\nabla(z-\pi_hz)\|_{L^2(I;L^2(\Omega))}+\|\nabla\pi_h(z-\Pi_{P_\tau}z)\|_{L^2(I;L^2(\Omega))}\nonumber\\
			&\le&
			\|\nabla(z-\pi_hz)\|_{L^2(I;L^2(\Omega))}+\|\nabla(z-\Pi_{P_\tau}z)\|_{L^2(I;L^2(\Omega))}\nonumber\\
			&\le& Ch \|z\|_{L^2(I;H^2(\Omega))}+C\tau^{\frac{1}{2}}(\|z\|_{L^2(I;H^2(\Omega))}+\|z\|_{H^1(I;L^2(\Omega))})\nonumber\\
			&\le&C(h+\tau^{\frac{1}{2}})(\|g\|_{L^2(I;L^2(\Omega))}+\|z_T\|_{H^1(\Omega)}),\label{key_11}
		\end{eqnarray}
		where we have used the $H^1$-stability of the $L^2$-projection $\pi_h$ and the estimate \eqref{estimate_15} in Lemma \ref{Lemma4}. Combining the two estimates \eqref{key_01} and \eqref{key_11} we finish the proof.
	\end{proof}

	\subsection{Error estimates for the optimal control problem}
	At first we prove a plain convergence for the solution of problem (\ref{min-J-h}) to that of problem (\ref{min-J}) as $|\sigma|:=\tau+h\rightarrow 0^+$.
	\begin{theorem}\label{Thm:plain_convergence}
		Let $\{\hat{q}_\sigma\}\subseteq \mathcal{M}(\bar I_c;L^2(\omega))$ be the set of optimal controls for the discrete optimal control problem (\ref{min-J-h}), and $\bar{u}_\sigma\in Y_\sigma$ be the unique discrete optimal state associated to $\{\hat{q}_\sigma\}$.  Let  $(\bar{q},\bar{u})\in \mathcal{M}(\bar I_c;L^2(\omega))\times X$ be the unique optimal pair of the continuous problem (\ref{min-J}), where $\bar{q}$ is the optimal control and $\bar{u}$ is the optimal state. Then  we obtain
		\begin{eqnarray}
			{q}_\sigma\stackrel{*}{\rightharpoonup} \bar{q}\in \mathcal{M}(\bar I_c;L^2(\omega))\quad\ \forall q_\sigma\in \{\hat{q}_\sigma\},\label{plain_conv_1}\\
			\|q_\sigma\|_{\mathcal{M}(\bar I_c;L^2(\omega))}\rightarrow \|\bar{q}\|_{\mathcal{M}(\bar I_c;L^2(\omega))}\quad\ \forall q_\sigma\in \{\hat{q}_\sigma\},\label{plain_conv_2}\\
			\|\bar{u}_\sigma-\bar{u}\|_{L^2(I;L^2(\Omega))}+\|(\bar{u}_\sigma-\bar{u})(T)\|_{L^2(\Omega)}\rightarrow 0,\quad \label{plain_conv_3}\\
			J_\sigma(q_\sigma)\rightarrow J(\bar{q})\quad\ \forall q_\sigma\in \{\hat{q}_\sigma\},\label{plain_conv_4}
		\end{eqnarray}
		when $|\sigma|\rightarrow 0^+$.
	\end{theorem}
	\begin{proof}
		The main ideas follow from  \cite[Theorem 4.9]{CasasClasonKunisch-2013} and \cite[Theorem 1.2]{HerbergHinzeSchumacher2020}, see also \cite[Theorem 3.5]{CasasClasonKunisch-2012}. Similar to Theorem \ref{Theorem3-5}, since $q_\sigma$ is optimal, we can easily show that the sequence $\{q_\sigma\}$ is uniformly bounded in $\mathcal{M}(\bar I_c;L^2(\omega))$ with respect to $\sigma$. 
		Then there exists a subsequence, still denoted by $\{ q_\sigma\}$, such that $q_\sigma\stackrel{*}{\rightharpoonup} \tilde q$ in $\mathcal{M}(\bar I_c;L^2(\omega))$ for some $\tilde q$ when $\sigma\rightarrow 0^+$. Below, we show that $\tilde q=\bar q$.
		
		Let $u_{\tilde{q}}$ be the solution of equation \eqref{PDE_state} with $q$ replaced by $\tilde q$. We first show that
		\begin{eqnarray}
			\|u_{\tilde q}-\bar u_\sigma\|_{L^2(I;L^2(\Omega))} +\|( u_{\tilde{q}}-\bar u_\sigma)(T)\|_{L^2(\Omega)}\rightarrow 0\quad \mbox{as}\ |\sigma|\rightarrow 0.\label{plain:aux_1}
		\end{eqnarray} 
		In fact, from the triangle inequality we have
		\begin{eqnarray}
			\|u_{\tilde q}-\bar u_\sigma\|_{L^2(I;L^2(\Omega))} +\|( u_{\tilde{q}}-\bar u_\sigma)(T)\|_{L^2(\Omega)}&\le& \|u_{\tilde q}- u_{ q_\sigma}\|_{L^2(I;L^2(\Omega))}+\|(u_{\tilde q}- u_{ q_\sigma})(T)\|_{L^2(\Omega)}
			\nonumber\\
			&+&
			\|u_{ q_\sigma}-\bar u_\sigma\|_{L^2(I;L^2(\Omega))}+\|(u_{ q_\sigma}-\bar u_\sigma)(T)\|_{L^2(\Omega)}.
		\end{eqnarray} 
		Applying Proposition \ref{Thm:continuity} to the first two terms and Theorem \ref{Thm:state_error_1} to the latter two terms on the right-hand side of the above estimate yields the result. 
		
		Next, there holds
		\begin{eqnarray}
			j(\tilde q)\leq \liminf\limits_{|\sigma|\rightarrow 0}j_\sigma(q_\sigma)\leq \limsup\limits_{|\sigma|\rightarrow 0}j_\sigma(q_\sigma)\leq \limsup\limits_{|\sigma|\rightarrow 0}j_\sigma(\Lambda_\sigma \bar{q})=j(\bar{q}),\label{plain:aux_3}
		\end{eqnarray}
		where in the first inequality we have used the weakly-$^*$ lower semicontinuity of the cost functional $j$,
		in the third inequality we used the optimality of $q_\sigma$, while in the last equality we have used (\ref{inter_8}) and \eqref{plain:aux_1}. 
		Therefore, $\tilde q$ is also optimal, so $\tilde q=\bar q$ since $\bar q$ is unique, i.e., $q_\sigma\stackrel{*}{\rightharpoonup} \bar q\in \mathcal{M}(\bar{I}_c;L^2(\omega))$, then $\bar u=u_{\tilde{q}}$ by the unique solvability of the state equation, which proves (\ref{plain_conv_1}) and  (\ref{plain_conv_3}).  Obviously, (\ref{plain_conv_4}) can be concluded from (\ref{plain:aux_3}). Lastly, (\ref{plain_conv_2}) follows from (\ref{plain_conv_3}) and (\ref{plain_conv_4}).
	\end{proof}
	
	A second convergence result concerns the convergence order of the objective functional.
	\begin{theorem}\label{Thm:obj_error}
		Let $q_\sigma\in \{\hat{q}_{\sigma}\}\subseteq \mathcal{M}(\bar{I}_c;L^2(\omega))$ be any optimal control to the discrete  problem (\ref{min-J-h}) and $\bar{q}\in \mathcal{M}(\bar I_c;L^2(\omega))$ be the unique optimal control for the continuous  problem (\ref{min-J}). Then there exists a constant $C>0$, independent of $\sigma$, such that
		\begin{eqnarray}
			|j(\bar{q})-j_\sigma(q_\sigma)|\leq C(h+\tau^{1\over 2}).\label{estimate_J}
		\end{eqnarray}
	\end{theorem}
	\begin{proof}
		The proof uses the approach of \cite[Theorem 5.1]{CasasClasonKunisch-2013}, see also \cite[Theorem 4.1]{CasasClasonKunisch-2012}. It follows from the optimality of $\bar{q}$ and $q_\sigma$ that
		\begin{eqnarray}
			j(\bar q)-j_\sigma(\bar q)\leq j(\bar q)-j_\sigma(q_\sigma)\leq j(q_\sigma)-j_\sigma(q_\sigma),\nonumber
		\end{eqnarray}	
		which means that
		\begin{eqnarray}
			|j(\bar q)-j_\sigma(q_\sigma)|\leq\max\{|j(\bar q)-j_\sigma(\bar q)|,  |j(q_\sigma)-j_\sigma(q_\sigma)|\}.\label{Thm:obj_auxi_1}
		\end{eqnarray}
		Now it remains to estimate the two terms on the right-hand side.
		
		For arbitrary $\tilde q\in \mathcal{M}(\bar I_c;L^2(\omega))$, we denote by $u_{\tilde q}$ and $u_\sigma(\tilde q)$ the unique solutions to problems (\ref{PDE_state}) and (\ref{state_h}), respectively. Then we obtain from Theorem \ref{Thm:state_error_1} that
		\begin{eqnarray}
			\|u_{\tilde q}-u_\sigma(\tilde q)\|_{L^2(I;L^2(\Omega))}+\|(u_{\tilde q}-u_\sigma(\tilde q))(T)\|_{L^2(\Omega)}\leq C(h+\tau^{1\over 2}).\nonumber
		\end{eqnarray}
		It is straightforward to show that
		\begin{eqnarray}
			|j(\tilde q)-j_\sigma(\tilde q)|&\leq& {1\over 2}\Big|\|u_{\tilde q}-u_d\|_{L^2(I;L^2(\Omega))}^2- \|u_\sigma(\tilde q)-u_d\|_{L^2(I;L^2(\Omega))}^2\Big|\nonumber\\
			&&+ {\beta\over 2}\Big|\|u_{\tilde q}(T)-u_T\|_{L^2(\Omega)}^2-\|u_\sigma(\tilde q)(T)-u_T\|_{L^2(\Omega)}^2\Big|\nonumber\\
			&\leq&C(\|u_{\tilde q}\|_{L^2(I;L^2(\Omega))}+\|u_\sigma(\tilde q)\|_{L^2(I;L^2(\Omega))}+\|u_d\|_{L^2(I;L^2(\Omega))})\|u_{\tilde q}-u_\sigma(\tilde q)\|_{L^2(I;L^2(\Omega))}\nonumber\\
			&&+C(\|u_{\tilde q}(T)\|_{L^2(\Omega)}+\|u_\sigma(\tilde q)(T)\|_{L^2(\Omega)}+\|u_T\|_{L^2(\Omega)})\|(u_{\tilde q}-u_\sigma(\tilde q))(T)\|_{L^2(\Omega)}\nonumber\\
			&\leq&C(h+\tau^{1\over 2}),\nonumber
		\end{eqnarray}
		where we have used Lemma \ref{Lm:state_h_stability} and Theorem \ref{Thm:existence_state}. By setting $\tilde q=q$ and $\tilde q=q_\sigma$ in the above error estimate we finish the proof by considering (\ref{Thm:obj_auxi_1}).
	\end{proof}
	
	The last convergence result is about the approximation of the state equation.
	\begin{theorem}\label{Thm:state_error_2}
		Let $\bar u\in L^2(I;L^2(\Omega))$ be the optimal state of the continuous optimal control problem (\ref{min-J}) and $\bar u_\sigma\in Y_\sigma$ be the discrete optimal state of the discrete optimization problem (\ref{min-J-h}). Then there exists a constant $C$, independent of $\sigma$, such that
		\begin{eqnarray}
			\|\bar u-\bar u_\sigma\|_{L^2(I;L^2(\Omega))}^2+\beta\|(\bar u-\bar u_\sigma)(T)\|_{L^2(\Omega)}^2\leq C(h+\tau^{1\over 2}).\label{estimate_state}
		\end{eqnarray}
	\end{theorem}
	\begin{proof}
		In order to obtain the above estimate \eqref{estimate_state}, we first introduce two auxiliary variables. The first one is the finite element approximation to the state equation \eqref{PDE_state} with the optimal control $\bar{q}$: Find $\hat u_\sigma\in Y_\sigma$ such that
		\begin{equation}
			A(\hat u_\sigma,v_\sigma)=\int_I(f,v_\sigma)dt+\langle \bar q,v_\sigma\rangle_{\bar I_c\times \omega}+(u_0,v_\sigma(0))\quad\ \forall v_\sigma\in P_\sigma,\nonumber
		\end{equation}
		while the second is the finite element approximation to the adjoint equation \eqref{adjoint_equation} with the optimal state $\bar u$: Find $\hat\varphi\in P_\sigma$ such that
		\begin{equation}
			A(\omega_\sigma,\hat \varphi_\sigma)=\int_I(\bar u-u_d,\omega_\sigma)dt+\beta(\bar u(T)-u_T,\omega_\sigma(T))\quad\  \forall \omega_\sigma\in Y_\sigma.\nonumber
		\end{equation}
		
		Taking $p=\hat q_\sigma$ in the continuous optimality condition \eqref{optimal_condition_2} and $p=\bar q$ in the discrete optimality condition \eqref{disctere_condition_2}, where $\hat q_\sigma$ is any optimal control for the discrete optimization problem \eqref{discret_optimal_reduced}, then adding them up we obtain 
		\begin{eqnarray}
			0&\le& \langle\hat q_\sigma-\bar q,\bar\varphi-\bar\varphi_\sigma \rangle_{\bar I_c\times \omega}\nonumber\\
			&=& \langle\hat q_\sigma-\bar q,\bar\varphi-\hat\varphi_\sigma \rangle_{\bar I_c\times \omega}+\langle\hat q_\sigma-\bar q,\hat\varphi_\sigma-\bar\varphi_\sigma \rangle_{\bar I_c\times \omega}\nonumber\\
			&=& \langle\hat q_\sigma-\bar q,\bar\varphi-\hat\varphi_\sigma \rangle_{\bar I_c\times \omega}+A(\bar u_\sigma-\hat u_\sigma, \hat\varphi_\sigma-\bar\varphi_\sigma)\nonumber\\
			&=& \langle\hat q_\sigma-\bar q,\bar\varphi-\hat\varphi_\sigma \rangle_{\bar I_c\times \omega}+(\bar u-\bar u_\sigma,\bar u_\sigma-\hat u_\sigma)+\beta((\bar u-\bar u_\sigma)(T),(\bar u_\sigma-\hat u_\sigma)(T))\nonumber\\
			&=& \langle\hat q_\sigma-\bar q,\bar\varphi-\hat\varphi_\sigma \rangle_{\bar I_c\times \omega}+(\bar u-\bar u_\sigma,\bar u_\sigma-\bar u)+\beta((\bar u-\bar u_\sigma)(T),(\bar u_\sigma-\bar u)(T))\nonumber\\
			&&+(\bar u-\bar u_\sigma,\bar u-\hat u_\sigma)+\beta((\bar u-\bar u_\sigma)(T),(\bar u-\hat u_\sigma)(T)).\nonumber
		\end{eqnarray}
		Therefore, there holds
		\begin{eqnarray}
			&&\|\bar u-\bar u_\sigma\|_{L^2(I;L^2(\Omega))}^2+\beta\|(\bar u-\bar u_\sigma)(T)\|_{L^2(\Omega)}^2\nonumber\\
			&\leq&2 \langle\hat q_\sigma-\bar q,\bar\varphi-\hat\varphi_\sigma \rangle_{\bar I_c\times \omega}
			+\|\bar u-\hat u_\sigma\|^2_{L^2(I;L^2(\Omega))}+\beta\|(\bar u-\hat u_\sigma)(T)\|^2_{L^2(\Omega)}\nonumber\\
			&\le&2\|\hat q_\sigma-\bar q\|_{\mathcal{M}(\bar{I}_c;L^2(\omega))}\|\bar\varphi-\hat\varphi_\sigma\|_{C(\bar{I}_c;L^2(\Omega))}
			+\|\bar u-\hat u_\sigma\|^2_{L^2(I;L^2(\Omega))}+\beta\|(\bar u-\hat u_\sigma)(T)\|^2_{L^2(\Omega)}\nonumber\\
			&\le&C(h+\tau^{1\over 2}),\nonumber
		\end{eqnarray}
		where we have used Theorems \ref{Thm:state_error_1}  and  \ref{estimate_adjoint}.
		This finishes the proof.
	\end{proof}

	{\bf Acknowledgement} The first author was supported by the Strategic Priority Research Program of Chinese Academy of Sciences under grant No. XDB41000000, the National Key Basic Research Program (2022YFA1004402) and the National Natural Science Foundation of China under grant 12071468.
	
	\bibliographystyle{abbrv}

\end{document}